\documentclass[12pt, reqno]{amsart}

\usepackage{esint}
\usepackage{bbm}
\usepackage{amstext}
\usepackage{amsthm}
\usepackage{amsmath}
\usepackage{amssymb}
\usepackage{latexsym}
\usepackage{amsfonts}
\usepackage{graphicx}
\usepackage{color}

\usepackage[pagebackref,hypertexnames=false, colorlinks, citecolor=black, linkcolor=blue, urlcolor=red]{hyperref}

\newcommand{\cz}{Calder\'{o}n--Zygmund\ }

\newcommand{\ci}[1]{_{ {}_{\scriptstyle #1}}}

\newcommand{\ti}[1]{_{\scriptstyle \text{\rm #1}}}

\newcommand{\R}{\mathbb{R}}
\newcommand{\E}{\mathbb{E}}
\newcommand{\Z}{\mathbb{Z}}
\newcommand{\C}{\mathbb{C}}

\newcommand{\F}{\mathbb{F}}
\newcommand{\cL}{\mathcal{L}}

\newcommand{\bI}{\mathbf{I}}

\newcommand{\bW}{\mathbf W}
\newcommand{\bw}{\mathbf w}
\newcommand{\bV}{\mathbf V}
\newcommand{\bv}{\mathbf v}

\newcommand{\dd}{\mathrm{d}}

\newcommand{\cF}{\mathcal{F}}
\newcommand{\cX}{\mathcal{X}}

\newcommand{\cD}{\mathcal{D}}
\newcommand{\cR}{\mathcal{R}}

\newcommand{\1}{\mathbf{1}}

\newcommand{\fT}{\mathfrak{T}}

\newcommand{\fS}{\mathfrak{S}}

\newcommand{\fdot}{\,\cdot\,}
\newcommand{\wt}{\widetilde}

\newcommand{\La}{\langle }
\newcommand{\Ra}{\rangle }
\newcommand{\rk}{\operatorname{rk}}
\newcommand{\Ch}{\operatorname{Ch}}

\newcommand{\ran}{\operatorname{Ran}}

\newcommand{\tr}{\operatorname{tr}}

\newcommand{\be}{\begin{equation}}
\newcommand{\ee}{\end{equation}}

\newcommand{\I}{\mathbf{I}}

\renewcommand{\labelenumi}{(\roman{enumi})}

\newcounter{vremennyj}

\newcommand\cond[1]{\setcounter{vremennyj}{\theenumi}\setcounter{enumi}{#1}\labelenumi\setcounter{enumi}{\thevremennyj}}

\numberwithin{equation}{section}

\newtheorem{thm}{Theorem}[section]
\newtheorem{lm}[thm]{Lemma}
\newtheorem{cor}[thm]{Corollary}

\newtheorem*{prop*}{Proposition}

\theoremstyle{remark}
\newtheorem{rem}[thm]{Remark}
\newtheorem*{rem*}{Remark}
\newtheorem{defin}[thm]{Definition}
\newtheorem*{defin*}{Definition}

\newtheorem{exa}[thm]{Example}
\newtheorem*{exa*}{Example}

\begin{document}
\title[Two weight estimates with matrix measures]{Two weight estimates with matrix measures for well localized operators}

\author[K. Bickel]{Kelly Bickel$^{\dagger}$}
\address{Kelly Bickel, Department of Mathematics\\
Bucknell University\\
701 Moore Ave\
Lewisburg, PA 17837, USA}
\email{kelly.bickel@bucknell.edu}
\thanks{$\dagger$ Research supported in part by National Science Foundation 
DMS grant \#1448846.}

\author[A. Culiuc]{Amalia Culiuc}
\address{Amalia Culiuc, School of Mathematics\\ Georgia Institute of Technology\\ 686 Cherry Street\\ Atlanta, GA 30332-0160 USA }
\email{amalia@math.gatech.edu}

\author[S. Treil]{Sergei Treil$^{\star}$}
\thanks{$^{\star}$ Research supported  in part by National Science Foundation DMS grants \#1301579, 1600139 .}
\address{Sergei Treil, Department of Mathematics \\ Brown University \\ Providence, RI 02912 USA}
\email{treil@math.brown.edu}

\author[B. D. Wick]{Brett D. Wick$^\ddagger$}
\address{Brett D. Wick, Department of Mathematics\\
Washington University in St. Louis\\
One Brookings Drive\\
 St. Louis, MO 63130-4899, USA}
\email{wick@math.wustl.edu}
\thanks{$\ddagger$ Research supported in part by National Science Foundation
DMS grant \#1500509.}
\maketitle

\begin{abstract}
In this paper, we give necessary and sufficient conditions for weighted $L^2$ estimates with matrix-valued measures of well localized operators.   Namely, we seek estimates of the form:
\[
\| T(\mathbf{W} f)\|_{L^2(\mathbf{V})} \le C\|f\|_{L^2(\mathbf{W})} 
\]
where $T$ is formally an integral operator with additional structure, $\mathbf{W}, \mathbf{V}$ are matrix measures, and the underlying measure space possesses a filtration.  The characterization we obtain is of Sawyer-type; in particular we show that certain natural testing conditions obtained by studying the operator and its adjoint on indicator functions suffice to determine boundedness.  Working in both the matrix weighted setting and the setting of measure spaces with arbitrary filtrations requires novel modifications of a T1 proof strategy; a particular benefit of this level of generality is that we obtain polynomial estimates on the complexity of certain Haar shift operators.
\end{abstract}

\setcounter{tocdepth}{1}
\tableofcontents
\setcounter{tocdepth}{3}

\setcounter{section}{-1}
\section{Introduction}

In this paper, we give necessary and sufficient conditions for two weight $L^2$ estimates with matrix-valued measures of the so-called \emph{well localized operators}. We seek estimates of the form:
\[
\| T(\bW f)\|\ci{L^2(\bV)} \le C\|f\|\ci{L^2(\bW)}, 
\]
where $T$ is formally an integral operator with additional structure and $\bW, \bV$ are matrix measures. The main examples we have in mind are  \emph{Haar shifts} 
and their different generalizations, considered in the matrix weighted spaces. For details concerning matrix measures, generalized Haar shifts, and well localized operators, see Sections 
\ref{Setup1}--\ref{Setup3}.

Our main results, Theorems \ref{well-loc-rel} and \ref{t:well-loc-est-02},  basically say that Sawyer-type testing conditions are necessary and sufficient for the boundedness of well localized operators. In other words, for the boundedness of such operators, it is sufficient to check the estimates of the operator and its adjoint on characteristic functions of cubes. 

One of the main motivations for the paper is the matrix $A_2$ conjecture.   This is an important open question in the matrix-weighted setting, which asks whether for a \cz operator $T$ 
\[  
\left \| T \right \|_{L^2(W) \rightarrow L^2(W)} \le C [W]_{A_2};
\]
here $C$ depends on $T$ but not $W$, and  
\[\big[W\big] _{A_2} := \sup_{I} \left \| \left( |I|^{-1} \int_I W(x) dx \right)^{1/2} \left( |I|^{-1} \int_I W(x)^{-1} dx \right)^{1/2} \right \|^2 < \infty\]
is the $A_2$ characteristic of $W$. For previous work on this problem, see \cite{bpw16, bw16, IKP, ps15}. Currently, the best known dependence of the norm of $T$ on the $A_2$  characteristic of $W$ in the matrix setting is $[W]_{A_2}^{3/2}$, which was established, first by the third author and collaborators, \cite{TPNV} and then by the second author and collaborators using some techniques from the preprint \cite{TPNV} coupled with a new method, \cite{FYC}.  In the scalar case, the first solution for particular cases (martingale multipliers, Hilbert Transform, Riesz Transforms, etc.) and the first solution for all \cz operators \cite{th12} used two weight techniques to get one weight estimates. In particular, in \cite{th12} a two weight estimate of Haar shifts, similar to a scalar version of our Theorem \ref{band-rel}, was one of the ingredients of the proof. 

The proof of the main results follows along the lines of \cite{NTV} and is outlined in Section \ref{mainproof}. The main part of the operator is estimated by a corresponding weighted paraproduct, see Section \ref{paraproduct}, and the estimate of the paraproduct is done using the Carleson Embedding Theorem. For the case of matrix-valued measures, one needs a matrix-weighted version of this theorem, with matrix-valued weights both in the domain and in the target space. Such a theorem appeared only recently in \cite{ct15}, making this paper possible.  


In this paper, we adapt our main result Theorem \ref{t:well-loc-est-02} to Haar shifts with a finite number of terms, which allows us to simplify the testing requirements. Indeed, Theorem \ref{band-rel} reduces the matrix $A_2$ conjecture for this class of operators to establishing a single testing condition and its dual. Earlier, similar results were obtained (only in the scalar case) in \cite{lt16, ttv15} with significant extra work from the results of \cite{NTV} or by modifying the proofs from \cite{NTV}. 

This paper not only establishes two-weight theorems in the dyadic matrix weighted setting, but also considers the problem in a much more general situation and establishes better testing bounds. Indeed,  we treat the case of very general filtrations, not just the standard dyadic one and so, our cubes are allowed to have an arbitrary number of children. This requires us to slightly alter the definition of well localized and be more careful when estimating the ``easier'' part of our operator, see Section \ref{mainproof}.  As a benefit of working in this level of generality, we are able to get better estimates and stronger results than the ones in \cite{NTV} or \cite{bw15}, even in the case of scalar measures. Indeed, our arguments give polynomial dependence on the complexity of the Haar shift (or band of the well localized operator), while 
the results from \cite{NTV},  \cite{bw15} only give
exponential dependence.  While the theorems we obtain also require one to test on a slightly more complicated class of functions, this additional condition can be removed if we  consider generalized Haar shifts and assume that measures satisfy a joint $A_2$ condition,  
 see Theorem \ref{band-rel}.  Thus, a dedicated reader will find the paper interesting even if they restrict themselves only to the scalar setting and to the setting when the underlying filtration is the standard dyadic lattice. Finally, we should mention that related matrix weighted results for different operators and with different estimates appear in \cite{bw15, k1, k2}. 

\section{Setup: matrix-valued measures and weighted estimates} 
\label{Setup1}
\subsection{Atomic filtered spaces}
Let $(\mathcal{X}, \cF, \sigma)$ be a sigma-finite measure space with a filtration $\{\cF_n\}$, that is, a sequence of increasing sigma-algebras $\cF_n\subset \cF$. Here, $\cF$ is the smallest sigma-algebra containing $\cup \cF_n.$ We make the assumption that $\cF_n$ is atomic, meaning that there exists a countable collection $\mathcal{D}_n$ of disjoint sets $Q$ of finite measure (which we call atoms or cubes) with the property that every set of $\cF_n$ is a union of atoms (cubes) $Q\in \mathcal{D}_n$. 

Denote by $\cD$ the collection of all atoms,  $\displaystyle\mathcal{D}=\cup_{n\in \mathbb{Z}}\mathcal{D}_n$. 
A set $Q$ could belong to multiple generations $\cD_n$, so atoms $Q\in\cD_n$ should formally be represented as pairs $(Q,n)$. However, to simplify notation, we will suppress the dependence on $n$ and write $Q$ instead of $(Q,n)$; if the generation (or rank) $n$ is needed, it will be represented by $\rk Q$, i.e.~if $Q$ stands for the atom $(Q,n)$ then $n=\rk Q$. The inclusion $R\subset Q$ for atoms is understood as set inclusion together with the inequality $\rk R \ge \rk Q$. In particular, for any $r\in \Z$, $\Ch^r Q$ stands for the collection of atoms $R\subset Q$ with $\rk R=r+\rk Q$. For $r=1$, we write $\Ch Q$ and avoid the superscript. 

For a measurable set $E$, 
we will often use the notation $|E|=\sigma(E)$.

\subsection{Examples}
\begin{exa} One motivating example is the standard dyadic filtration in $\R^d$ with Lebesgue measure. 
For $n\in \Z$, let 
\[ 
\cD_n := \{ 2^{-n}\left( (0, 1]^d+ k \right) : k\in \Z^d \} 
\]
be the collection of dyadic cubes with side length $2^{-n}$. Then each $\cF_n$ is the $\sigma$-algebra generated by $\mathcal{D}_n$, and $\cF$ is the Borel $\sigma$-algebra. 
In this example, we do not have atoms of different ranks coinciding as sets. 
\end{exa}

The standard dyadic filtration also leads to more interesting examples.

\begin{exa} Consider a measurable $\cX \subset \R^d$, again endowed with Lebesgue measure. For each $n \in \Z$, define the collection of atoms $\cD_n$  as the collection of all non-empty intersections $Q\cap\cX$, where $Q$ runs over all dyadic cubes of side length $2^{-n}$ from the previous example. 

If, for example, $\cX=Q_0=(0, 1]^d$, then $Q_0\in \cD_n$ for all $n\le 0$, and we have cubes of different ranks coinciding as sets.  Taking more complicated $\cX,$ we can have more complicated structures of atoms and their ranks.  We can make this example even more complicated by letting the underlying measure $\sigma$ be an arbitrary Radon measure. 
\end{exa}



\subsection{Matrix-valued measures}
\label{s:mvm}
Let $\cF_0$ be the collection of sets  $E\cap F$ where  $E\in\cF$ and $F$ is a finite union of atoms.  
A $d\times d$ matrix-valued measure $\bW$ on $\cX$ is a countably additive function on $\cF_0$ with values in the set of  non-negative linear operators on $\F^d$, where $\F$ is either $\C$ or $\R$. Equivalently, $\bW=(\bw_{j,k})_{j,k=1}^d$ is a $d\times d$ matrix whose entries $\bw_{j,k}$ are (possibly signed or even complex-valued) measures that are finite on atoms, such that for any $E\in\cF_0,$ the matrix $(\bw_{j,k}(E))_{j,k=1}^d$ is positive semidefinite. Note that the measure $\bW$ is always finite on atoms. Given such a measure $\bW$ and measurable functions $f=(f_1, f_2, \ldots, f_d)^T$ and $g=(g_1, g_2, \ldots, g_d)^T$ on $\cX$ with values in $\F^d,$ we can define the integrals
\[
 \int_\cX \Bigl\La \dd \bW f, g \Bigr\Ra_{\F^d} :=\sum_{j,k=1}^d \int_\cX f_k \overline g_j \dd\bw_{j,k}, \qquad \int_\cX \dd \bW f,
\]
where the second integral is the vector whose $j$th coordinate is given by $\sum_{k=1}^d \int_\cX f_k \dd\bw_{j,k}$.

The weighted space $L^2(\bW)$ is the space of  measurable, $\F^d$-valued functions on $\cX$ satisfying
\[
\|f\|\ci{L^2(\bW)}^2 := \int_\cX \Bigl\La \dd \bW f, f\Bigr\Ra\ci{\F^d} <\infty. 
\]

Readers not comfortable with matrix-valued measures can always, without loss of generality, restrict themselves to working with absolutely continuous measures and matrix-valued functions. Namely, it is an easy corollary of the non-negativity of the matrix measure $\bW$ that all of the measures $\bw_{j,k}$ are absolutely continuous with respect to the trace measure $\bw:= \tr \bW := \sum_{k=1}^d \bw_{k,k}$. Therefore, we can write $\dd\bW = W \dd\bw$, where $W$ is a $\textbf{w}$-a.e.~positive semidefinite $d \times d$ matrix-valued function on $\cX$ and 
\[
\int_\cX \dd\bW f = \int_\cX Wf \dd \bw, \qquad \int_\cX \Bigl\La \dd \bW f, g \Bigr\Ra_{\F^d} = \int_\cX \Bigl\La  W f, g \Bigr\Ra_{\F^d} \dd \bw.
\]
The matrix-valued function $W$ is called the density of $\bW$ with respect to $\bw.$

\subsection{Weighted estimates with matrix weights}
This paper deals with two weight estimates of discrete ``integral'' operators $T$ that are represented (at least formally) as $Tf(x) = \int_{\cX} K(x,y) f(y) \dd\sigma(y)$, where the kernel $K(x,y)=k(x,y) \otimes \textbf{I}_d$, for a scalar-valued kernel $k(x,y)$. We are interested in estimates of the form 
\begin{align}
\label{est-01}
\| T(\bW f)\|\ci{L^2(\bV)} \le C\|f\|\ci{L^2(\bW)} 
\end{align}
with matrix-valued measures $\bV$ and $\bW$. Here $T(\bW f)$ is defined for the integral operator $T$ by 
\begin{align}
\label{T_W}
T(\bW f)(x) = \int_\cX \dd \bW (y) K(x,y) f(y) = \int_\cX K(x,y) W(y) f(y) \dd \bw(y),
\end{align}
where $W$ is the density of $\bW$ with respect to the scalar trace measure $\bw$. 
We will use the symbol $T\ci\bW$ for the operator $f\mapsto T(\bW f)$ and the symbol $T_\bw$ for the operator $f\mapsto T(f\bw)$, where
\begin{align}
\label{T_w}
T_\bw f(x) := \int_\cX K(x,y) f(y) \dd\bw(y). 
\end{align}
This operator $T_\bw$ is defined for both scalar-valued functions and functions with values in $\F^d$; we will use the same notation for both cases, although formally in the latter case, we should write $T_\bw\otimes \bI_d$. 
If $\dd\bW = W\dd\bw$ and $\dd \bV=V\dd \bv$ for any scalar measures $\bw, \bv$ defined on $\mathcal{F}$ and positive semidefinite functions $W, V$, we can rewrite estimate \eqref{est-01} as
\begin{align}
\label{est-02}
\|V^{1/2} T_\bw W^{1/2} f\|\ci{L^2(\bv)} \le C \|f\|\ci{L^2(\bw)}. 
\end{align}
A particularly interesting case is when the measures $\bV$ and $\bW$ are absolutely continuous with respect to the underlying measure $\sigma$. Then we can write $\dd\bW = W\dd \sigma$ and $\dd \bV=V\dd \sigma$ and in \eqref{est-02}, we can just take $\bv=\bw=\sigma$. 




\section{Setup: generalized band operators and Haar shifts} \label{Setup2}
\subsection{Expectations and martingale differences.}
Let us introduce some notation and terminology. We call a measurable function $f$ \emph{locally integrable} if it is integrable on every atom $Q\in\cD$. 
 For an atom $Q$ and a locally integrable function $f$, we denote by $\La f\Ra\ci Q$ its average (with respect to the underlying measure $\sigma$)
\[
\La f \Ra\ci Q := \sigma(Q)^{-1} \int_Q f d\sigma, 
\]
with the convention that $\La f \Ra\ci Q=0$ if $\sigma(Q)=0$.  Define the averaging operator (expectation) $\E\ci Q$ by
\begin{align}
\label{AVO}
\E\ci Q f := \La f\Ra\ci Q \1\ci Q  
\end{align}
and the martingale difference operator $\Delta\ci Q$ by
\begin{align}
\label{MDO}
 \Delta\ci Q := \sum_{R\in\Ch Q} \E\ci R  - \E\ci Q. 
\end{align}
Note that $\E\ci Q$ and $\Delta\ci Q$ are orthogonal projections in $L^2(\sigma)$ and that the subspaces generated by the $\Delta\ci Q$ are orthogonal to each other.   We think of $\E\ci Q$, $\Delta\ci Q$ as operators in Lebesgue spaces ($\E\ci Q f$, $\Delta\ci Q f$ are defined $\sigma$-a.e.), so if for atoms $Q_{1}\subset Q_2$ we have $\sigma(Q_2\setminus Q_1) =0$, then $\Delta\ci{Q_2}=0$.

\subsubsection{Generalized band operators}
To the collection $\cD$, associate a tree structure where each $Q$ is connected to the elements of the collection $\Ch Q$. Given this tree, let $d\ci{\text{tree}}(Q,R)$ denote the ``tree distance'' between atoms $Q$ and $R$, namely, the number of edges of the shortest path connecting $Q$ and $R$. If $Q$ and $R$ share no common ancestors, then $d\ci{\text{tree}}(Q,R) = \infty.$ 

The operators of interest possess a band structure related to this tree distance, as defined below. These operators are called \emph{generalized band operators} because they generalize the band operators studied in \cite{NTV,bw15}.
\begin{defin}
A bounded operator $T:L^2(\sigma) \rightarrow L^2(\sigma)$ is  a \textit{generalized band operator of radius $r$} if $T$ can be written as 
\begin{align}
\label{GenBand-01}
T=\sum\ci{j,k=1}^2\sum\ci{Q,R \in \cD}P\ci R^j T\ci{R,Q}^{j,k}P\ci Q^k \qquad T\ci{R,Q}^{j,k}: P\ci Q^k L^2(\sigma) \rightarrow P\ci R^j L^2(\sigma),
\end{align}
where each $T^{j,k}\ci{R,Q}$ is a bounded operator on  $L^2(\sigma)$ satisfying $T^{j,k}\ci{R,Q}=0$ if $d\ti{tree}(R,Q)>r$ and for any $Q \in \cD$, the projections are defined as  $P\ci Q^1:=\Delta\ci Q$ and $P\ci Q^2:=\E\ci Q$.

In general, convergence is in the weak operator topology with respect to some ordering of the pairs $Q,R$.  We typically assume that the sum in \eqref{GenBand-01} has only finitely many nonzero terms once we collect the blocks in groups as in \eqref{BandToHaar}.
\end{defin}

Then each block $\wt T^{j,k}\ci{R,Q} := P\ci R^j T\ci{R,Q}^{j,k}P\ci Q^k$ is a bounded operator in $L^2(\sigma)$ and can be  represented as an integral operator with kernel $K\ci{R,Q}^{j,k}$. The  kernel $K\ci{R,Q}^{j,k}$ can be computed as follows:~for $y\in Q$ let $Q_y\in\Ch Q$ be the unique child of $Q$ containing $y$. Defining
\begin{align}
\label{CanKern}
K\ci{R, Q}^{j, k} (x,y) := 
\begin{cases} |Q_y|^{-1} \left(\wt T\ci{R,Q}^{j,k} \1\ci{Q_y} \right)(x), & y\in Q \\
0 & y\notin Q, \end{cases}
\end{align}
one can easily show that 
\begin{align*}
\left(\wt T^{j,k}\ci{R,Q} f\right)(x) = \int_\cX K\ci{R, Q}^{j, k} (x,y)  f(y) \dd\sigma(y)
\end{align*}
for all functions $f\in L^2(\sigma)$  such that $\1\ci Q f$ is supported on a finite union of cubes $S\in\Ch Q$.  One  can similarly show that the formula holds for functions of the form $a\textbf{1}_Q$, $a \in \mathbb{F}.$ Then since $\wt T^{j,k}\ci{R,Q} f =\wt T^{j,k}\ci{R,Q} ( f \1 \ci Q)$ for all $f \in L^2(\sigma),$ Lemma $\ref{dense}$  implies that the above integral formula for $\wt T^{j,k}_{R,Q}$ holds on a dense set of $ L^2(\sigma)$ and therefore, defines $\wt T^{j,k}_{R,Q}$.

Note also that the kernel $K\ci{R, Q}^{j, k}$ is supported on $R\times Q$ and is constant on sets $R'\times Q'$, for $R'\in\Ch R$, $Q'\in\Ch Q$. 

Because the operators $P^k_Q$  are orthogonal projections in $L^2(\sigma)$ and hence are self-adjoint, $T$ being a generalized band operator of radius $r$ implies that its adjoint $T^*$ is also a generalized band operator of radius $r$. Here are several examples of generalized band operators:

\begin{exa}
For a numerical sequence $a=\{a\ci Q\}\ci{Q\in\cD}$ define the ``dyadic'' operator 
$T_a$ on $L^2(\sigma)$  by
\[
T_af=\sum_{Q \in \cD} a\ci Q\E\ci Qf . 
\]
Trivially, $T_a$ is a generalized band operator of radius $r=0$ as long as $T_a$ is bounded on $L^2(\sigma).$
\end{exa} 

\begin{rem*}
For a sequence $|a|:=\{|a\ci Q|\}\ci{Q\in\cD}$ one can easily see that the pointwise estimate
\begin{align*}
\left | \left(T_a f \right)  \right |(x) \le \left(T_{|a|} |f| \right)(x)  \qquad \forall x\in\cX
\end{align*}
holds.  So, in the scalar case, the two weight estimates for $T_a$ follow from the two weight estimates for $T_{|a|}$. Thus, in the scalar case, the operators with all $a\ci Q\ge 0$ (the so-called positive dyadic operators) play a special role in weighted estimates. 

In the case of  matrix-valued measures, it is not clear that the  weighted estimates of $T_{|a|}$ imply the corresponding estimates for $T_a$ (we suspect that this is not true),  so we do not reserve any special place for the positive dyadic operators. 
\end{rem*}

\begin{exa} For $r\in\Z_+$ and $b$ a locally integrable function, define the paraproduct $\Pi=\Pi_b^r$ of order $r$ on $L^2(\sigma)$ by 
\begin{align}
\label{Para-ex}
\Pi f=\sum_{Q \in \cD} \E\ci Q f \sum_{R\in\Ch^r Q}  \Delta\ci R b = \sum_{Q \in \cD}  \sum_{R\in\Ch^r Q}   \left( \Delta\ci R M_b \ \E\ci Q \right) f, 
\end{align}
where $M_b$ is multiplication by $b$. Clearly, as long as $\Pi$ is bounded on $L^2(\sigma)$, it is a generalized band operator of radius $r$.

\end{exa}
\begin{rem*}
Since $\Pi f$ is defined by the sum of an orthogonal series, the  convergence of the sum defining $\Pi$ in the weak operator topology implies its unconditional convergence in the strong operator topology. 
\end{rem*}

\begin{exa}
\label{ex:HaarShift-01}
A Haar shift of complexity $(m,n)$ is an operator $T:L^2(\sigma)\to L^2(\sigma)$ defined by
\begin{align}
\label{HaarShift-01}
T = \sum_{Q\in \cD} \  \sum_{\substack{R\in\Ch^n(Q),\\ S\in\Ch^m(Q)}} \Delta\ci R T\ci{R,S}\Delta\ci S, \qquad 
T\ci{R,S} : \Delta\ci S L^2(\sigma) \to \Delta\ci R L^2(\sigma),
\end{align}
where  for each bounded operator $\wt T\ci{R,S} := \Delta\ci R T\ci{R,S}\Delta\ci S,$ its  canonical kernel $K\ci{R,S}$  (defined by \eqref{CanKern} and supported on $R\times S$) satisfies the estimate 
\begin{align}
\label{HaarShiftNormalization-01}
\|K\ci{R,S}\|_\infty \le |Q|^{-1},
\end{align}
which in this paper, means $| K\ci{R,S} (x,y)| \le  |Q|^{-1}$ for all $x,y \in \cX.$
If $T$ is a Haar shift of complexity $(m,n)$, then trivially its adjoint $T^*$ is also a Haar shift of complexity $(n, m)$.  Any Haar shift of complexity $(m,n)$ (in fact, any bounded operator given by \eqref{HaarShift-01}) is a generalized band operator of radius $r=m+n$. \end{exa}

\begin{rem*}
An  operator defined by \eqref{HaarShift-01} is  bounded if and only if all blocks $T\ci Q$ 
\begin{align*}
T\ci Q:= \sum_{\substack{R\in\Ch^n(Q),\\ S\in\Ch^m(Q)}} \Delta\ci R T\ci{R,S}\Delta\ci S
\end{align*}
are uniformly bounded: in this case the series in \eqref{HaarShift-01} converges unconditionally (independently of the ordering) in the strong operator topology.

Note, that the normalization condition \eqref{HaarShiftNormalization-01} implies that $\|T\ci Q\| \le 1$. Indeed, \eqref{HaarShiftNormalization-01} implies that the block $T\ci Q$ can be represented as an integral operator with kernel $K\ci Q$ (supported on $Q\times Q$) satisfying $\|K\ci Q\|_\infty \le |Q|^{-1}$, so $\|K\ci Q\|\ci{L^2(Q\times Q)} \le 1$.  

\end{rem*}

The concept of Haar shifts can be generalized. 

\begin{defin}
\label{df:GenHaarShift}
A generalized Haar shift of complexity $(m,n)$ is an operator $T:L^2(\sigma)\to L^2(\sigma)$ defined by
\begin{align}
\label{GenHaarShift-01}
T = \sum_{j,k=1}^2\sum_{Q\in \cD} \  \sum_{\substack{R\in\Ch^n(Q),\\ S\in\Ch^m(Q)}} P^j\ci R T\ci{R,S}^{j,k}P^k\ci S, \qquad 
T\ci{R,S}^{j,k} : P^k\ci S L^2(\sigma) \to P^j\ci R L^2(\sigma),
\end{align}
where each $T^{j,k}\ci{R,S}$ is bounded, the projections are $P\ci Q^1:=\Delta\ci Q$ and $P\ci Q^2:=\E\ci Q$, and the kernel $K\ci{R,S}^{j,k}$ of  $\wt T^{j,k}\ci{R,S} := P\ci R^j T\ci{R,S}^{j,k} P^k\ci S$ satisfies
\begin{align}
\label{HaarShiftNormalization-02}
\| K^{j,k}\ci{R,S} \|_\infty \le |Q|^{-1}.
\end{align}
We typically assume that the sum in \eqref{GenHaarShift-01} has only finitely many nonzero $Q$ terms. In general, convergence is in the weak operator topology with respect to some ordering of the pairs $R,S$.  
\end{defin}
It is convenient to present an alternate representation of a (generalized) Haar shift by  grouping the terms $\wt T\ci{R,Q}^{j,k}$. Namely, denoting 
\begin{align*}
T\ci{Q} = \sum_{j,k=1}^2  \sum_{\substack{R\in\Ch^n(Q),\\ S\in\Ch^m(Q)}} P^j\ci R T\ci{R,S}^{j,k}P^k\ci S
\end{align*}
(or taking the inner sum in \eqref{HaarShift-01} for regular Haar shifts),  we can represent a generalized Haar shift as $\sum_{Q\in\cD} T\ci Q$. Note that the  kernel $K\ci Q$ of the integral operator $T\ci Q$ is supported on $Q\times Q$, constant on $R\times S$, $R, S\in\Ch^{r+1} Q$, $r=\max\{m,n\}$. Since the sets $R\times S$, $R\in\Ch^n Q$, $S\in\Ch^m Q$ are disjoint, the kernel $K_{Q}$ 
satisfies the estimate
\begin{align}
\label{HaarShiftNormalization-03}
\| K\ci Q\|_\infty \le |Q|^{-1}
\end{align}
for Haar shifts, and the estimate  $\| K\ci Q\|_\infty \le 4 |Q|^{-1}$ for generalized ones. We need the constant ``$4$'' here because for each pair $R\in \Ch^n Q$, $ S\in\Ch^m Q,$ there are four operators $T\ci{R, S}^{j,k}$. This discussion motivates the following general object of study:

\begin{defin}
\label{d:BigHaarShift}
A \emph{generalized big Haar shift} of complexity $r$ is a bounded operator $T: L^2(\sigma) \rightarrow L^2(\sigma)$ defined by
\begin{align}
\label{BigHaarShift}
T=\sum_{Q\in\cD} T\ci Q, 
\end{align}
where each block $T\ci Q$ is an integral operator with kernel $K\ci Q$, where $K\ci Q$ is supported on $Q\times Q $, constant on $R\times S$ with $R, S\in\Ch^{r+1}Q,$ and  
satisfies the estimate \eqref{HaarShiftNormalization-03}.  If in addition, each block $T\ci Q$ and its adjoint $T^*\ci Q$ annihilate constants $\1\ci Q$, we will call the operator simply a \emph{big Haar shift} of complexity $r$, without  the word \emph{generalized}. 
Finally, if an operator $T$ admits the above representation but without the estimate \eqref{HaarShiftNormalization-03}, we will say that the operator $T$ \emph{has the structure of a (generalized) big Haar shift} of complexity $r$. 
\end{defin}

\begin{rem}
\label{r:band-Haar}
It is easy to see that a (generalized) band operator of radius $r$ has the structure of a (generalized) big Haar shift of complexity $r$. To see that, we can just define 
\begin{align}
\label{BandToHaar}
T\ci Q := \sum_{R\in\Ch^r Q} \sum_{\substack{S\in\cD(Q)\\ \rk S \le \rk R}} \sum_{j,k=1}^2 \wt T\ci{R,S}^{j,k}+  \sum_{S\in\Ch^r Q} \sum_{\substack{R\in\cD(Q)\\ \rk R < \rk S}} \sum_{j,k=1}^2 \wt T\ci{R,S}^{j,k}.
\end{align}
Moreover, if the kernels $K\ci{R,S}^{j,k}$ of the blocks $\wt T\ci{R,S}^{j,k}$ admit the estimate $\| K\ci{R,S}^{j,k}\|_\infty\le 1/4 |Q|^{-1}$ (or the estimate $\| K\ci{R,S}\|_\infty\le |Q|^{-1}$ for kernels of the blocks $\wt T\ci{R,S}$ for the case of a band operator), then the operator $T$ is a (generalized) big Haar shift. 
\end{rem}

\section{Setup: weighted martingale differences and well localized operators} \label{Setup3}
\subsection{Weighted martingale differences}
For the matrix measure $\bW$ (or $\bV$) discussed above, one can define the $\bW$-weighted expectation $\E\ci Q^{\bW}$ and the martingale difference $\Delta\ci Q^{\bW}$ by
\begin{align}
\label{E^W-01}
\E\ci Q^{\bW} f  & = \La f \Ra\ci{Q}^{\bW} \1\ci Q , \qquad \La f \Ra\ci{Q}^{\bW} :=\bW(Q)^{-1}\left(\int_Q d\bW f \right)
\intertext{and} 
\label{Delta^W}
 \Delta\ci Q^{\bW}&=\displaystyle\sum\ci{R\in \Ch Q}\E\ci R^{\bW}-\E\ci Q^{\bW} 
\end{align}
respectively, for all atoms $Q \in \mathcal{D}.$  

Initially, this definition only makes sense if $\bW(Q)$ is invertible. However if $\bW(Q)$ is not invertible, we can interpret $\bW(Q)^{-1}$  as the Moore--Penrose pseudoinverse of $\bW(Q).$ Here, the Moore--Penrose pseudoinverse of a matrix $A$ is the unique matrix $A^{+}$ defined as follows: on $\ker A$, it is the zero operator and on $(\ker A)^{\perp} = \text{Ran } A$, it is the inverse of $A|_{\text{Ran }A}$. For example, if $\bW(Q)=0$, then $\bW(Q)^{-1}=0.$ 

The standard computations for $\E\ci Q^{\bW} $ and $ \Delta\ci Q^{\bW}$ still work in this general case because $\int_Q \dd\bW f \in \ran \bW(Q)$. To see this, write $\bW=W(x) d\textbf{w},$ where $\textbf{w}$ is the trace measure of $\textbf{W}$, and observe that a vector $e \in \text{Ker} \bW(Q)$ if and only if $\textbf{1}_Q e$ equals the zero function in $L^2(\bW).$ One can use this to show $\int_Q \dd\bW f  \perp \text{Ker} \bW(Q)$ and using $\bW(Q)$ self-adjoint, conclude 
$\int_Q \dd\bW f \in \ran \bW(Q)$.


Then, it is not too hard to see that $\E\ci Q^\bW$ is the orthogonal projection in $L^2(\bW)$ onto the subspace of constants  $\{\1\ci Qe: e\in \F^d\}$. To prove this fact, one needs to use properties of the Moore--Penrose pseudoinverse to show that $\1_Q e= \E\ci Q^\bW( \1_Q e)$ in $L^2(\bW)$ for all $Q \in \mathcal{D}$, $e \in \mathbb{F}^d.$ It can also be shown that $\E\ci Q^\bW  \Delta\ci Q^\bW  = \Delta\ci Q^\bW \E\ci Q^\bW =0$, that $\Delta\ci Q^\bW$ is an orthogonal projection, and that the subspaces generated by  $\Delta\ci Q^\bW$ and  $\Delta\ci R^\bW$ are orthogonal whenever $Q \ne R.$


%
%

\subsection{Well localized operators} To state and prove the main results, it is convenient to introduce the formalism of \emph{well localized operators} between weighted spaces, rather than work directly with operators that have the structure of a (generalized) big Haar shift.  Earlier, we defined the operator $T\ci \bW$, $T\ci\bW f := T(\bW f)$ as the integral \eqref{T_W}, provided that the integral is well defined. But to verify boundedness, we only need to know the bilinear form of the operator on a dense set. 

Let $\mathcal{L}$ denote the set of finite linear combinations of functions of the form $\1\ci Qe$, with $Q\in\cD$ and $e\in\F^d$. By Lemma \ref{dense}, this set is dense in $L^2(\bW)$ and so, it suffices to know how to compute
\begin{align*}
\La T\ci {\bW} \1\ci Q e, \1\ci R v\Ra\ci{L^2( \bV)} & = \iint_{\cX \times \cX} \Big\La \dd\bW(y) K(x, y) \1\ci Q(y) e, \dd\bV(x) \1\ci R(x) v\Big\Ra\ci{\F^d}\\
&= \iint_{\cX \times \cX} \Big\La W(y) K(x, y)  \1 \ci Q(y) e, V(x) \1\ci R(x) v\Big\Ra\ci{\F^d} \dd\bw(y)\dd\bv(x) 
\end{align*}
for all $Q, R\in \cD$  and $e, v\in\F^d$. To be precise, we say:

\begin{defin} An operator $T\ci \bW$ \emph{acts formally from $L^2(\bW)$ to $L^2(\bV)$} if the bilinear form 
\begin{align}
\label{T_W-formal}
\La T\ci \bW \1\ci Q e,\1\ci R v\Ra\ci{L^2(\bV)}
\end{align}
is well defined for all $Q, R\in \cD$ and all $e,v\in \F^d$. Then, the formal adjoint $T^*\ci \bV$ is given by 
\[
\La T\ci \bW \1\ci Q e,\1\ci R v\Ra\ci{L^2(\bV)} = \La  \1\ci Q e, T\ci{\bV}^*\1\ci R v\Ra\ci{L^2(\bW)} .
\]
As part of the definition, we also require a very weak continuity property, namely that
\begin{align}
\label{w-cont}
\La T_{\bW} \1\ci Q e , \1\ci R v\Ra\ci{L^2(\bV)} 
& = \sum_{S\in\Ch Q} \La T_{\bW} \1\ci S e , \1\ci R v\Ra\ci{L^2(\bV)}  \\
\notag
& = \sum_{S\in\Ch R} \La T_{\bW} \1\ci Q e , \1\ci S v\Ra\ci{L^2(\bV)} ;
\end{align}
this property is nontrivial only if $Q$ or $R$ have infinitely many children. 
\end{defin}

Consider the set $\cL$ of finite linear combinations of functions of the form $\1\ci Qe$, with $Q\in\cD$ and $e\in\F^d$. 
If the bilinear form \eqref{T_W-formal} is defined, then
\begin{align}
\label{T_W-formal-01}
\La T\ci\bW f, g\Ra\ci{L^2(\bV)} = \La  f, T\ci\bV^* g\Ra\ci{L^2(\bW)}
\end{align}
is well defined for all $f, g\in\cL$.
Since $\Delta^\bW\ci Q f \in\cL$ for $f\in\cL$, the expression $\La T\ci \bW \Delta\ci Q^\bW f, \Delta\ci R^\bV g \Ra\ci{L^2(\bV)}$ is also well defined for all $f, g\in\cL$. Thus  the expression $\Delta\ci R^\bV T\ci \bW \Delta\ci Q^\bW$ is well defined, in the sense that its bilinear form is well defined for $f,g\in\cL.$  

\begin{defin}
\label{d:loc-op}
An operator 
$T\ci \bW$ acting formally from $L^2(\bW)$ to $L^2(\bV)$ is said to be \emph{localized} if for all $e,v \in \mathbb{F}^d$,
\[\La  T\ci \bW \1\ci Q e, \1\ci R v \Ra\ci{L^2(\bV)}=0\]
whenever $Q, R \in \mathcal{D}$ share no common ancestors. 
\end{defin}

\begin{defin}
\label{df:LTr}
An operator $T\ci \bW$ acting formally from $L^2(\bW)$ to $L^2(\bV)$ is called $r$-\emph{lower triangular} if for all $R, Q \in \cD$ and $e \in \mathbb{F}^d$, 
\[
\Delta\ci R^{\bV} T\ci\bW \1\ci Q e=0
\]
if either
\begin{enumerate}
\item $R\not\subset Q$ and $\rk R\geq r+ \rk Q$, or
\item $R\not\subset Q^{(r+1)}$ and $\rk R\geq \rk Q-1$;
\end{enumerate}
here $Q^{(r+1)}$ is the order $r+1$ ancestor of $Q$, i.e.~$Q^{(r+1)}$ is the unique atom with  $Q \subset Q^{(r+1)}$ and $\rk Q^{(r+1)}=\rk Q-(r+1)$.
%
\end{defin}

As mentioned earlier, $T\ci \bW$ is defined via a bilinear form. So, in Definition \ref{df:LTr}, the statement $\Delta\ci R^{\bV} T\ci\bW \1\ci Q e=0$ should be interpreted as 
\[ \left \langle \Delta\ci R^{\bV} T\ci\bW \1\ci Q e, g \right \rangle_{L^2(\bV)} = \left \langle  T\ci\bW \1\ci Q e, \Delta\ci R^{\bV} g \right \rangle_{L^2(\bV)} =0,\]
for all $g \in \mathcal{L}.$ More generally, for any $f \in \mathcal{L}$, the notation  $T\ci \bW f$ or $\Delta\ci R^{\bV}  T\ci \bW f$ should be interpreted in terms of the bilinear form.

\begin{defin}
An operator $T\ci \bW$ acting formally from $L^2(\bW)$ to $L^2(\bV)$ is  said to be \emph{well localized with radius $r$} if it is localized and if both  $T\ci \bW$ and its formal adjoint $T\ci \bV^*$ are $r$-lower triangular. 
\end{defin}   

\begin{rem}
\label{r:vsNTV}
This definition is very similar to the definition of well localized operators from \cite{NTV}, with two exceptions. First, there was a typo in \cite{NTV}, and in the language of this paper, the definition in \cite{NTV}  only required that $\rk R\ge \rk Q$  in condition \cond2 of the above Definition \ref{df:LTr}. This was a typo; the inequality $\rk R\ge \rk Q$ is not sufficient to get the results in \cite{NTV}. For more details, see the discussion in \cite{bw15}. 

The other difference, which is more essential, is that in \cite{NTV}, it was not required  for the operator $T\ci \bW$ to be \emph{localized} in the sense of the above Definition \ref{d:loc-op}. 
In this paper, by requiring our operators to be localized, we are able to get better estimates than those in \cite{NTV}. In particular, we do not require any bounds on the number of children of a cube $Q\in\cD$. In \cite{NTV}, it was assumed that each cube had at most $N$ children for some fixed $N \in \mathbb{N}$, and the estimates depended on this bound.  

Since all of the examples we have in mind (all of them were presented earlier) give rise to localized operators, we included this requirement in the definition of well localized operators. Thus, we are able to get better estimates than those in \cite{NTV}, even for the case of scalar measures. 
\end{rem}

\subsection{From band operators to well localized operators}



Now we will show that if $T$ has the structure of  a generalized big Haar shift of complexity  $r$, see Definition \ref{d:BigHaarShift}, and if $T$ also satisfies several additional assumptions, then the operator $T\ci\bW$, $T\ci\bW f = T(\bW f)$ is a well localized operator of radius $r$. 

We assume that we only have finitely many terms $T\ci Q$ in the representation \eqref{BigHaarShift} and that each block $T\ci{ Q}$ is represented by an integral operator with a bounded kernel. Note that the latter assumption is always true if all $Q \in \mathcal{D}$ have finitely many children; for the generalized big Haar shifts, it is just postulated. 

The above two assumptions imply that the bilinear form \eqref{T_W-formal-01} is well defined for $f, g\in \cL$. Moreover, the facts that the kernel is bounded and $\bW, \bV$ are finite on atoms can be used to show that $T\ci\bW$ and its formal adjoint satisfy the weak continuity property \eqref{w-cont}. Thus $T\ci\bW$ is well defined as an operator acting formally from $L^2(\bW)$ to $L^2(\bV)$. In fact, since $T$ is an integral operator with a bounded, compactly supported kernel, it can be shown that the bilinear form \eqref{T_W-formal-01} is well defined for all $f\in L^2(\bW)$, $g\in L^2(\bV)$, so in fact $T\ci\bW$ is a bounded operator $L^2(\bW)\to L^2(\bV)$. 

\begin{lm}\label{lem:well}
Let $T$ have the structure of a generalized big Haar shift of complexity $r$, satisfying the assumptions above. 
Then for matrix-valued measures $\bW$ and $\bV,$ the operator $T\ci\bW$, $T\ci\bW f= T(\bW f)$, acting formally from $L^2(\bW)$ to $L^2(\bV)$ 
is a well localized operator of radius $r$.  
\end{lm}

\begin{proof}
The fact that the operator $T\ci\bW$ is localized, see Definition \ref{d:loc-op}, is obvious.  Now, we will show that $T\ci \bW$ is $r$-lower triangular; then by the symmetry, we obtain the same result for $T^*\ci\bV$. 

Observe that if $Q \in \cD, e\in\F^d,$ the assumptions on $T$ imply that the function $T(\bW \1\ci Q e )$ is well-defined and in $L^2(\bV).$ Then to prove that $T\ci \bW$ is $r$-lower triangular, it suffices to show that outside of $Q^{(r+1)}$, $T(\bW \1\ci Q e )$ is constant on cubes $R$ with $\rk R\ge \rk Q -1$,  and that outside of $Q,$ it is constant on cubes $R$ with  $\rk R\ge \rk Q +r$. 

Let us analyze when $T\ci{S}(\bW \1\ci Q e)$ can be non-zero and how it behaves in that case.  First, observe that $T\ci S (\bW \1\ci Q e)$ is non-zero outside of $Q$ only if $Q\subsetneq S$. Since $\rk S\le \rk Q-1$ for  $ Q\subsetneq S$, the condition $\rk R \ge \rk Q + r$ implies that 
\begin{align}
\label{rkR>rkS}
\rk R \ge \rk Q + r \ge \rk S+ 1 + r. 
\end{align}
We know  that the kernel of $T\ci S$ is constant on sets $S'\times S''$ with $S', S'' \in\Ch^{r+1} S$, and therefore, $T\ci S (\bW \1\ci Q e)$ is constant on cubes $R$ such that
\begin{align*}
\rk R \ge \rk S + r + 1. 
\end{align*}
So, if $R\cap Q=\varnothing$ and $\rk R \ge \rk Q + r$, we can conclude from \eqref{rkR>rkS} that $T\ci S (\bW \1\ci Q e)$ is constant on $R$. 
Similarly, if $T\ci S (\bW \1\ci Q e)$ does not vanish outside of $Q^{(r+1)}$, then $Q^{(r+1)}\subsetneq S$, so
\begin{align*}
\rk S \le \rk Q^{(r+1)} - 1 = \rk Q - (r+1) -1 =\rk Q - r - 2,
\end{align*}
or equivalently 
\begin{align*}
\rk S + r + 1 \le \rk Q-1.
\end{align*}
 The condition $\rk R \ge \rk Q -1$ then implies  that  
\begin{align}
\label{rkR>rkS-01}
\rk R \ge \rk Q - 1 \ge \rk S+ r +1 . 
\end{align}
But, as we discussed above,  $T\ci S (\bW \1\ci Q e)$ is constant on cubes $R$ such that $\rk R \ge \rk S + r + 1$, so \eqref{rkR>rkS-01} implies that  outside of $Q^{(r+1)},$ the function $T\ci S (\bW \1\ci Q e)$ is constant on cubes $R$ with $\rk R \ge\rk Q-1$.   
\end{proof}

\section{Main results}
\noindent
\subsection{Estimates of well localized operators}
For a cube $Q\in\cD,$ define $D_Q^{\bW,r}$ to be the collection of functions 
\begin{align*} D_Q^{\bW,r} := \left \{ f_Q = \sum_{R \in \Ch^r Q} \Delta\ci R^{\bW} f : \ \  f\in \cL \right\}.
\end{align*}
Given this definition, we can state our first main result.
\begin{thm}\label{well-loc-rel}
Let $T\ci {\bW}$ be a well localized operator of radius $r$ acting formally from $L^2(\bW)$ to $L^2(\bV)$. Then $T\ci \bW$ extends to a bounded operator from $L^2(\bW)$ to $L^2(\bV)$ if and only if the following conditions 
\begin{enumerate}
\item $ \|\1\ci Q T\ci \bW\1\ci Q\ e\|\ci{L^2(\bV)}\le \fT_1 \|\1\ci Q e\|\ci{L^2(\bW)}$ for all $e\in\F^d$;
\item $\|\1\ci Q T\ci \bW f\ci Q\|\ci{L^2(\bV)}\leq \fT_2\|f\ci Q\|\ci{L^2(\bW)}$ for all $f_Q\in D_Q^{\bW, r}$; 
\end{enumerate}
and their dual counterparts  (corresponding conditions for $T\ci\bV^*$ with $\bV$ and $\bW$ interchanged) hold for all $Q\in \cD$.
Furthermore,
\[
\| T\ci{\bW} \|\ci{L^2(\bW) \rightarrow L^2(\bV) }\leq \Bigl( C(d)^{1/2} +1/2\Bigr)(\fT_1+\fT_1^*)+  (r+1)^{1/2}(\fT_2+\fT_2^* );
\]
here $\fT_{1}^*$, $\fT_{2}^*$ are the constants from the duals to the testing conditions \textup{\cond1}, \textup{\cond2} respectively, and $C(d)$ is the constant from the Matrix Carleson Embedding Theorem (Theorem \ref{cet} below). 

Moreover, for the best possible bounds $\fT_k$, $\fT_k^*$ we trivially have 
\begin{align}
\label{fT<norm}
\fT_k, \, \fT_k^* \le \| T\ci{\bW} \|\ci{L^2(\bW) \rightarrow L^2(\bV) }\qquad k=1,2.
\end{align}

\end{thm}


\begin{rem*}
In the case when each $Q\in\cD$ has at most $N$ children ($N<\infty$), condition \cond2 follows from the testing condition $\|T\ci\bW \1\ci Q e\|\ci{L^2(\bV)} \le \fT_3\|\1\ci Qe\|\ci{L^2(\bW)}$. In this case, one can estimate $\fT_2 \le C(r, N) \fT_3$ and obtain similar estimates for the dual condition. 
That is very similar to the approach used in \cite{NTV}. 
\end{rem*}

Condition \cond1 in Theorem \ref{well-loc-rel} can be slightly relaxed. Given a well localized operator $T\ci\bW$ and cube $Q \in \mathcal{D}$, define its \emph{truncation} $T\ci\bW^Q$ by
\begin{align}
\label{T_W^Q}
T\ci\bW^Q f = \sum_{R\in\cD(Q)} \Delta\ci R^\bV T\ci\bW f, 
\end{align}
and similarly for the dual $T^*\ci\bV$. 

\begin{thm}
\label{t:well-loc-est-02}
Let $T\ci {\bW}$ be a well localized operator of radius $r$ acting formally from $L^2(\bW)$ to $L^2(\bV)$. Then $T\ci \bW$ extends to a bounded operator from $L^2(\bW)$ to $L^2(\bV)$ if and only if the  conditions 
\begin{enumerate}
\item $ \| T\ci \bW^Q\1\ci Q e\|\ci{L^2(\bV)}\le \fT_1 \|\1\ci Q e\|\ci{L^2(\bW)}$ for all $e\in\F^d$;

\item $\| T^Q\ci \bW f\ci Q\|\ci{L^2(\bV)}\leq \fT_2\|f\ci Q\|\ci{L^2(\bW)}$ for all $f_Q\in D_Q^{\bW, r}$, 

\end{enumerate}
their dual counterparts  (corresponding conditions for $T\ci\bV^*$ with $\bV$ and $\bW$ interchanged) 
and the following weak type estimate 
\begin{enumerate}
\setcounter{enumi}{2}
\item $\left| \La T\ci \bW \1\ci Q e, \1\ci Q v\Ra\ci{L^2(\bV)}\right| \le \fT_3 
\|\1\ci Q e\|\ci{L^2(\bW)}  \|\1\ci Q v\|\ci{L^2(\bV)}$ for all $e, v\in \F^d$, 
\end{enumerate}
hold for all $Q\in \cD$. Furthermore,
\[
\| T\ci{\bW} \|\ci{L^2(\bW) \rightarrow L^2(\bV) }\leq C(d)^{1/2}(\fT_1+\fT_1^*)+  (r+1)^{1/2}(\fT_2+\fT_2^* ) +  \fT_3; 
\]
here $\fT_{1}^*$, $\fT_{2}^*$ are the constants from the duals to the testing conditions \textup{\cond1}, \textup{\cond2} respectively, and $C(d)$ is the constant from the Matrix Carleson Embedding Theorem (Theorem \ref{cet} below). 
Moreover, for the best possible bounds $\fT_k$, $\fT_k^*$ we trivially have 
\begin{align}
\label{fT<norm-01}
\fT_k, \, \fT_k^* \le \| T\ci{\bW} \|\ci{L^2(\bW) \rightarrow L^2(\bV) }\qquad k=1,2, 3.
\end{align}
\end{thm}
There is no dual condition to \cond3 in this theorem because this condition is \emph{self-dual}. 
Note that Theorem \ref{well-loc-rel} follows immediately from Theorem \ref{t:well-loc-est-02}, because trivially for all $f\in\cL,$
\[
\| T\ci \bW^Q f \|\ci{L^2(\bV)} \le  \| \1\ci Q T\ci \bW f \|\ci{L^2(\bV)}. 
\]
Note also that condition \cond1 of Theorem \ref{well-loc-rel} implies condition \cond3 of Theorem \ref{t:well-loc-est-02}, with the trivial estimate for the corresponding bounds
\[
\fT_3 \le (\fT_1+ \fT_1^*)/2;
\]
here $\fT_3$ is the bound from Theorem \ref{t:well-loc-est-02}, and $\fT_1$, $\fT_1^*$ are the bounds from condition \cond1 and its dual in Theorem \ref{well-loc-rel}. 
\begin{rem}
\label{r:weak}
Condition \cond3 of Theorem \ref{t:well-loc-est-02} can be further relaxed. First, we do not need this condition to hold for all cubes $Q\in\cD$: it is sufficient if this condition holds for arbitrarily large cubes $Q$, meaning that for any $Q_0\in \cD$ one can find $Q\in\cD$, $Q_0\subset Q$ for which \cond3 holds. 

Secondly, if for any $Q_0\in\cD$ we have for the increasing sequence of cubes $Q_n$, where $Q_{n+1}$ is the parent of $Q_n$,  that $\bW(Q_n) \ge \alpha_n\I$, $\alpha_n\to +\infty$ and similarly for $\bV$, then condition \cond3 can be removed from Theorem \ref{t:well-loc-est-02}. 
\end{rem}

\subsection{Applications to estimates of Haar shifts}
While conditions \cond1 from Theorems \ref{well-loc-rel} and \ref{t:well-loc-est-02} are pretty standard testing conditions, and condition \cond3 from Theorem \ref{t:well-loc-est-02} is the standard weak boundedness condition, condition \cond2 seems unnecessarily complicated. However, if the $d\times d$ matrix measures $\bV$, $\bW$ satisfy the two weight matrix $A_2$ condition 
\begin{align}
\label{A_2}
\sup_{Q\in\cD} |Q|^{-2}\| \bV(Q)^{1/2} \bW(Q)^{1/2}\|^2 =: [\bV, \bW]\ci{A_2}<\infty, 
\end{align}
and the operator $T$ is a generalized big Haar shift with finitely many terms, see Definition \ref{d:BigHaarShift},  then condition \cond2 follows from a testing condition similar to \cond1 and the $A_2$ condition \eqref{A_2}.

Let us introduce some notation. For a generalized big Haar shift $T=\sum_{Q\in\cD} T\ci Q$ and cube $Q \in \cD$, define the operator $T^Q$ by
\begin{align}
\label{T^Q}
T^Q := \sum_{R\in\cD(Q)} T\ci R.
\end{align}
For a matrix measure $\bW,$ define the weighted version 
$(T^Q)\ci\bW$ of $T^Q$ by 
\begin{align*}
(T^Q)\ci\bW f = T^Q(\bW f).
\end{align*}
Note that $(T^Q)\ci\bW$ is different from $T^Q\ci\bW$ defined above in \eqref{T_W^Q}. 

\begin{lm}
\label{l:TestHaar}
Let $T$ be a generalized big Haar shift of complexity $r$ with finitely many terms, and let the $d\times d$ matrix measures $\bV$ and $\bW$ satisfy the matrix $A_2$ condition \eqref{A_2}. Assume that for all $Q\in\cD$
\begin{align}
\label{test-T^Q}
\| (T^Q)\ci\bW \1\ci Q e \|\ci{L^2(\bV)}   \le \fT \| \1\ci Q e \|\ci{L^2(\bW)},  \text{ for all } e\in \F^d.  
\end{align}

Then for all $Q\in\cD$
\begin{align}
\label{test-T_W^Q}
&\| T^Q\ci\bW \1\ci Q e \|\ci{L^2(\bV)}  \le \left( d^{1/2} r [\bV, \bW]\ci{A_2}^{1/2} +   \fT  \right) \|\1\ci Q e\|\ci{L^2(\bW)},  \text{ for all } e\in\F^d;\\
\label{test-T_W^Q-01}
&\| T^Q\ci\bW f_Q \|\ci{L^2(\bV)}  \le \left( d^{1/2}  (2r+1) [\bV, \bW]\ci{A_2}^{1/2} + \fT \right) \|f_Q \|\ci{L^2(\bW)},  \text{ for all }  f_Q \in D\ci Q^{\bW, r} .
\end{align}
Moreover, for all sufficiently large $Q\in\cD$
\begin{align}
\label{weak-02}
\| T\ci \bW \1\ci Q e \|\ci{L^2(\bV)} \le \fT \| \1\ci Q e \|\ci{L^2(\bW)}.
\end{align}
\end{lm}



Lemma \ref{l:TestHaar} implies that for a generalized big Haar shift $T$ of complexity $r$ with finitely many terms, the bounds in the testing conditions in Theorem \ref{t:well-loc-est-02} can be estimated as 
\begin{align*}
 \fT_1  & \le d^{1/2} r [\bV, \bW]\ci{A_2}^{1/2}  +\fT ,\\
\fT_2  &  \le d^{1/2} (2r+1) [\bV, \bW]\ci{A_2}^{1/2} +\fT ,\\
\fT_3 & \le \fT ,
\end{align*}
with the similar estimates for the dual bounds $\fT_{1}^*,\fT_{2}^*$. Note also  that  $\fT_3 \le\fT^*$, so $\fT_3\le(\fT+\fT^*)/2$.  Using these estimates and applying Theorem \ref{t:well-loc-est-02}, we get the following result. 

\begin{thm}\label{band-rel}
Let $T$ be a generalized big Haar shift  of complexity  $r$ with finitely many terms, and let the $d \times d$ matrix measures  $\bV$ and $\bW$ satisfy the $A_2$ condition \eqref{A_2}.   
Let 
\begin{enumerate}
\item $ \| (T^Q)\ci\bW \1\ci Q e \|\ci{L^2(\bV)}   \le \fT \| \1\ci Q e \|\ci{L^2(\bW)}$
for all $Q\in \cD$ and all vectors $e\in \F^d$,
\end{enumerate}
and let also the corresponding  condition for $T^*$ (with $\bV$ and $\bW$ interchanged) hold with constant $\fT^*$. 
Then,
\begin{align*}
\|T\ci\bW\|\ci{L^2(\bW) \rightarrow L^2(\bV) }\leq 
&\Bigl( C(d)^{1/2}  +(r+1)^{1/2} + 1/2 \Bigr) (\fT+\fT^*) \\
& +
2d^{1/2} \Bigl( C(d)^{1/2} r + (2r+1) (r+1)^{1/2}\Bigr) [\bV, \bW]\ci{A_2}^{1/2}\, ;
\end{align*}
here again $C(d)$ is the constant from the Matrix Carleson Embedding Theorem (Theorem \ref{cet}).
\end{thm}

\begin{rem} Under some additional assumptions on the filtration $\{\cF_n\}_{n\in\Z}$,  \cond1 of Theorem \ref{band-rel} (and its adjoint) is also necessary for the boundedness of the operator $T\ci\bW: L^2(\bW)\to L^2(\bV)$. For example, it is necessary if there exists $\kappa\in (0,1)$ such that for any $Q\in\cD$ and its parent $\hat Q$
\[
|Q|\le\kappa |\hat Q|.
\]
For the proof, see Lemma \ref{l:nec} below.  This condition holds for the standard dyadic lattice, and for any homogeneous lattice, but it is in fact much weaker than homogeneity of the lattice. 
%
\end{rem}

 Lemma \ref{l:TestHaar} will also be proved later in Section \ref{s:est-Haar}. 

\subsection{Remarks about norm dependence on the $A_2$ characteristic and complexity}
Let $T$ be a Haar shift  of complexity  $r$ with finitely many terms and let $W$ be an $A_2$ weight. Then, Theorem \ref{band-rel} could potentially be used to estimate the dependence of the norm of $T$ in the weighted space $L^2(W)$ on the $A_2$ characteristic $[W]\ci{A_2}$ of the weight $W$. Here, $[W]\ci{A_2}$ is the exactly the $A_2$ characteristic $[\bV, \bW]\ci{A_2}$ from  \eqref{A_2} with $\dd\bW =W\dd\sigma$, $\dd\bV=W^{-1}\dd\sigma$. 

In the case of scalar weights $w$, there is an estimate showing the norm of interest  \linebreak $\|T\|_{L^2(w) \rightarrow L^2(w)}$ depends linearly on $[w]\ci{A_2}$, and this estimate is optimal, see \cite{th12}. In the  matrix case, the best known estimate is $[W]\ci{A_2}^{3/2},$ which has been recently established independently by the third author and collaborators and the second author and collaborators. In any case, for Haar shifts with finitely many terms, Theorem \ref{band-rel} reduces the problem to finding the optimal estimate in the testing condition \cond1 and its dual. 

Theorem \ref{band-rel} can also be used to study the dependence of $\|T\|_{L^2(W) \rightarrow L^2(W)}$ on the complexity of $T$. For scalar $A_2$ weights $w$, the best known estimate for $\|T\|_{L^2(w) \rightarrow L^2(w)}$ grows linearly in the complexity $r$ of the Haar shift $T$, see \cite{t13}. It appears that in the scalar case, Theorem \ref{band-rel} gives the growth rate $r^{3/2}$ in terms of complexity, because testing constants similar to those in \cond1 and its dual are usually estimated by $C\cdot(r+1)[w]\ci{A_2}$. 

However, then the standard splitting trick would allow us to get linear in complexity growth. Namely, one can split the operator $T$ as $T=\sum_{k=0}^r T_k$ 
\[
T_k = \sum_{j\in\Z}\sum_{ \substack{ Q\in\cD \\ \rk Q = k + (r+1)j}} T\ci Q ; 
\]
then each $T_k$ is a  generalized big Haar shift of complexity zero, with respect to the rarefied  filtration given by $\sigma$-algebras  $\cF_{k+(r+1)n}$,  $ n\in\Z$. 

In the scalar case, with the dyadic filtration, estimates of testing bounds like $\fT$ and $\fT^*$ in terms of $[w]\ci{A_2}$ do not appear to change if we restrict to rarefied filtrations, see \cite{t13}.
As a result, we could potentially use these arguments to obtain an estimate of $\|T\|_{L^2(W) \rightarrow L^2(W)}$  that depends linearly on the complexity $r$.

\section{Weighted paraproducts and their estimates} \label{paraproduct}

The essential part of the proof of the main results is the estimate of the associated weighted paraproducts, which is presented in this section.  

\subsection{Weighted paraproducts}
Let $f= \1_Se$ be a characteristic function with $S \in \mathcal{D}$ and  $e\in \mathbb{F}^d$. Then, for each fixed $n \in \mathbb{Z}$, $f$ has the orthogonal decomposition
\begin{equation} \label{eqn:decomp}
f=\sum_{Q\in \cD, \rk Q\ge -n}\Delta\ci Q^{\bW} f +\sum_{Q\in \cD, \rk Q= -n} \E\ci Q^{\bW} f.
\end{equation}
To prove this equality, just observe that if $\rk Q \ge \rk S$, then $f \textbf{1}_Q= \E\ci Q^{\bW} f$ in $L^2(\bW)$. So, if $m \ge \rk S$, we can use the definition of $\Delta\ci Q^{\bW}$ to conclude
\[ f -   \sum_{\substack{ Q \in \mathcal{D} \\  m \ge \rk Q \ge -n}} \left(\Delta\ci Q^{\bW} f \right) -  \sum_{\substack{ Q \in \mathcal{D} \\ \rk Q = -n}} \left(\E^{\bW}_Q f \right)  =  f - \sum_{\substack{Q \in \mathcal{D} \\ \rk Q =m+1}} \left(\E^{\bW}_Q f\right)= 0\]
in $L^2(\bW).$ Letting $m \rightarrow \infty$ gives the desired result. By orthogonality, it follows that 
\[
\|f\|^2_{L^2(\bW)}=\sum_{Q\in \cD, \rk Q\ge -n} \left \|\Delta\ci Q^{\bW} f \right \|^2_{L^2(\bW)} +\sum_{Q\in \cD, \rk Q= -n} \left \| \E\ci Q^{\bW}f \right\|^2_{L^2(\bW)}.
\]
For an operator $T\ci\bW$ acting formally from $L^2(\bW) $ to $ L^2(\bV),$
define the  paraproduct $\Pi^\bW=\Pi^\bW\ci T$ of complexity $r$ as
\begin{align}
\label{para-01}
\Pi^{\bW}f=\sum_{Q\in \cD}\sum_{R\in \Ch^r Q} \Delta\ci R^{\bV} \left(T\ci \bW \E\ci Q^{\bW} f\right) 
= \sum_{Q\in \cD}\sum_{R\in \Ch^r Q} \Delta\ci R^{\bV} \left(T\ci \bW \La f \Ra\ci Q^{\bW} \1\ci Q\right);
\end{align}
in the situations we consider, one can show that the bilinear form of the paraproduct $\Pi^\bW$ is well defined for $f,g\in\cL$. Similarly for the adjoint $T\ci \bV^*$ of $T^\bW,$ define the paraproduct $\Pi^{\bV} =\Pi^\bV\ci {T^*}$ by
\[
\Pi^{\bV} g=\sum_{Q\in \cD}\sum_{R\in \Ch^r Q} \Delta\ci R^{\bW} \left(T\ci {\bV}^* \E\ci Q^{\bV} f\right).
\]
In later studies of these paraproducts, we will require the following simple lemma.

\begin{lm}
\label{l:T-para}
Let $T=T\ci\bW$ be a well localized operator of radius $r$, acting formally from $L^2(\bW)$ to $L^2(\bV)$.
Then for any cubes  $Q, S \in \mathcal{D}$ with $Q\subset S$ and for any $R\in\Ch^r Q$ and for all $e\in\F^d$, 
\begin{align*}
 \Delta\ci R^\bV T\ci\bW \1\ci Q e =\Delta\ci R^\bV T\ci\bW \1\ci S e. 
\end{align*}
\end{lm}

\begin{rem}
\label{r:T-para}
The above Lemma \ref{l:T-para} means that in the formula \eqref{para-01} for paraproducts, one can replace  $\La f \Ra\ci Q^{\bW} \1\ci Q$ by $\La f \Ra\ci Q^{\bW} \1\ci S$ with an arbitrary cube $S\supset Q$. So formally we can write in the right hand side of \eqref{para-01} the expression $\La f \Ra\ci Q^{\bW} \1$ instead of $\La f \Ra\ci Q^{\bW} \1\ci Q$, which looks more in line with the definition of the paraproduct in the scalar case. 

To make it even more similar to the scalar representation, we could use $T\ci\bW (\1\otimes\I\ci{\F^d})$ instead of $T\ci\bW\1$ (to apply the operator $T\ci\bW$ to a matrix-valued function, one just needs to apply it to each column), and write the paraproduct $\Pi^\bW$ as 
\begin{align}
\label{para-02}
\Pi^\bW f = \sum_{Q\in\cD} \sum_{R\in\Ch^r Q} 
\left( T\ci\bW (\1\otimes\I\ci{\F^d}) \right)  \La f \Ra\ci Q^{\bW},
\end{align}
which is an alternate way of writing \eqref{para-01}. The expression $T\ci\bW (\1\otimes\I\ci{\F^d}) $ should be understood as $T\ci\bW (\1\ci S\otimes\I\ci{\F^d}),$ where $S$ is an arbitrary cube with $Q\subset S$.  
\end{rem}

\begin{proof}[Proof of Lemma \ref{l:T-para}]
Take a cube $P\ne Q$, $\rk P = \rk Q$.  Since $T\ci\bW$ is $r$-lower triangular, 
\[
\Delta\ci R^\bV T\ci\bW \1\ci P e = 0, 
\]
for any cube $R\not\subset P$, $\rk R\ge \rk P + r$. In particular, that holds for $R\in\Ch^r Q$. 

Since for a cube $S\supset Q$ the set $S\setminus Q$ is a (countable) union of cubes $P$, $\rk P=\rk Q$, we conclude using the weak continuity property \eqref{w-cont} that for any $R\in\Ch^r Q$
\begin{align*}
\Delta\ci R^\bV T\ci\bW \1\ci{S\setminus Q} e = 0, 
\end{align*}
 which proves the lemma. 
\end{proof}

\begin{rem}
\label{r:para-01}
As one can see, in  the above proof we only used the fact that $T$ is $r$-lower triangular; more precisely,  only a part of the definition was used. 
\end{rem}

The following lemma states that the paraproducts $\Pi^{\bW}$ and $\Pi^{\bV}$  exhibit the same behavior as $T\ci \bW$ and $T^*\ci \bV$ respectively.
\begin{lm}
\label{lem:replacement}
Let $T\ci \bW$ be a well localized operator of radius $r$ acting formally from $L^2(\bW)$ to $L^2(\bV)$, and let $\Pi^{\bW} =\Pi\ci T^\bW$ be the paraproduct of complexity $r$ defined as above. 
Then for $Q, R \in \cD$
\begin{enumerate}
\item If $\rk R\leq r+\rk Q$, then 
\[\Delta^\bV\ci R\Pi^{\bW}\Delta^\bW\ci Q=0.\]
\item If $R\not\subset Q$, then 
\[\Delta^\bV\ci R\Pi^{\bW}\Delta^\bW\ci Q=0.\]
\item If $\rk R >r+\rk Q$, then
\[\Delta^\bV\ci R\Pi^{\bW}\Delta^\bW\ci Q=\Delta^\bV\ci R T\ci{\bW}\Delta^\bW\ci Q,\]
and in particular if $R\not \subset Q$, both sides of the equality are zero.
\end{enumerate}
\end{lm}

\begin{proof}
Using summation indices $Q'$ and $R'$, we have
\[
\Pi^{\bW} \Delta\ci Q^{\bW}=\sum_{Q'\in \cD}\sum_{R'\in \Ch^r(Q')} \Delta\ci {R'}^{\bV} \left(T\ci \bW \E\ci {Q'}^{\bW}  \Delta\ci Q^{\bW}\right),
\]
and since $\Delta\ci R^{\bV}$ is orthogonal to $\Delta\ci{R'}^{\bV}$ for all choices of $R'$ except for $R$,  we have
\[\Delta^\bV\ci R\Pi^{\bW}\Delta^\bW\ci Q= \Delta\ci R^{\bV}  T\ci \bW \E\ci {Q'}^{\bW}  \Delta\ci Q^{\bW}\]
where $Q' = R^{(r)}$ is the $r^{th}$ order ancestor of $R$. Notice that $\E\ci {Q'}^{\bW}  \Delta\ci Q^{\bW} \neq 0$ only if $Q'\subsetneq Q$. So, if $\rk R\leq r+\rk Q$ then $\rk R^{(r)} \le \rk Q,$ which implies $\E\ci {Q'}^{\bW}  \Delta\ci Q^{\bW} = 0$, and consequently 
\[\Delta^\bV\ci R\Pi^{\bW}\Delta^\bW\ci Q= 0,\]
proving the first statement.
Also, if $R\not\subset Q$, then $Q' = R^{(r)} \not \subset Q$. As above, this implies $\E\ci {Q'}^{\bW}  \Delta\ci Q^{\bW} = 0$, and consequently 
\[\Delta^\bV\ci R\Pi^{\bW}\Delta^\bW\ci Q= 0,\]
which proves the second statement.

To prove the third statement, assume  $\rk R >r+\rk Q$. If $R\not\subset Q$, we can use our previous result and the fact that $T_{\bW}$ is well localized to conclude:
\[\Delta^\bV\ci R\Pi^{\bW}\Delta^\bW\ci Q=0=\Delta^\bV\ci R T\ci{\bW}\Delta^\bW\ci Q.\]
It now suffices to consider the case $R\subset Q$. Recall that $Q' = R^{(r)}$. Since $Q \cap Q' \ne \emptyset$, we can look at ranks to conclude that 
 $Q' \subsetneq Q.$ Choose $\tilde{Q} \in \Ch Q$ with $Q' \subseteq \tilde{Q}$. Then, using the fact that $T_{\bW}$ is $r$-lower triangular, we have
 \[ \Delta\ci R^{\bV} T\ci \bW \Delta\ci Q^{\bW} =  \Delta\ci R^{\bV} T\ci \bW\left(  \sum_{S \in \Ch Q} \E\ci {S}^{\bW}  - \E\ci {Q}^{\bW}\right) =   \Delta\ci R^{\bV} T\ci \bW\left(   \E\ci {\tilde{Q}}^{\bW}  - \E\ci {Q}^{\bW} \cdot \textbf{1}_{\tilde{Q}} \right).\]
Using earlier arguments and $Q' \subsetneq Q$, we can write $ \Delta^\bV\ci R\Pi^{\bW}\Delta^\bW\ci Q $ as 
 \[ \Delta\ci R^{\bV}  T\ci \bW \E\ci {Q'}^{\bW}  \Delta\ci Q^{\bW}
=  \Delta\ci R^{\bV}  T\ci \bW \left(  \E\ci {\tilde{Q}}^{\bW} \cdot \textbf{1}_{Q'} - \E\ci {Q}^{\bW} \cdot \textbf{1}_{Q'} \right)  
=  \Delta\ci R^{\bV} T\ci \bW\left(   \E\ci {\tilde{Q}}^{\bW}  - \E\ci {Q}^{\bW} \cdot \textbf{1}_{\tilde{Q}} \right),
 \]
where the last equality follows by Lemma \ref{l:T-para}, completing the proof.
%
%
%
\end{proof}

\subsection{Estimates of the paraproducts}
The following theorem by the second and third authors from \cite{ct15} will be used to control the norms of the paraproducts:

\begin{thm}[The matrix weighted Carleson Embedding Theorem]\label{cet}
Let $\bW$ be a $d\times d$ matrix-valued measure and let $ \{A\ci Q\}_{Q \in \mathcal{D}}$ be a sequence of positive semidefinite $d\times d$ matrices indexed by $\mathcal{D}.$ Then the  following statements are equivalent:
\begin{enumerate}
\item $\displaystyle \sum_{Q\in \cD} \left\|  A^{1/2}\ci Q \int_Q \dd \bW f \right\|^2 \le A \|f\|\ci{L^2(\bW)}^2$
\item $\displaystyle \sum_{\substack{Q\in \cD(Q_0) 
}}   \bW ( Q ) A\ci Q\bW ( Q )\le B  \bW (Q_0 )$ for all $Q_0\in\cD$. 
\end{enumerate}
Moreover, for the best constants $A$ and $B$ we have $B\le A \le C(d) B$, where $C(d)$ is a constant depending only on the dimension $d$.
\end{thm}

\begin{rem*}
In \cite{ct15}, the authors obtained the value $C(d)=e\cdot d^3(d+1)^2$, where $e$ is the base of the natural logarithm. This might give the optimal asymptotic in terms of the dimension $d$, but it seems unlikely. 
\end{rem*}




Now, we bound the paraproducts as follows:
\begin{lm}  
\label{lem:parbdd} Let $\Pi^{\bW}$ be the paraproduct defined earlier and assume that the well localized operator $T_{\bW}$ satisfies the testing condition
\begin{align}
\label{test-03} 
 \sum_{\substack{  R\in\cD(Q) \\ \rk R \ge\rk Q +r}} \left \| \Delta\ci R^\bV T\ci \bW\1\ci Q e \right\|\ci{L^2(\bV)}^2\le \fT_1^2 \|\1\ci Q e\|\ci{L^2(\bW)}^2 
\end{align}
for all $Q \in \mathcal{D}$ and $e \in \mathbb{F}^d$. Then $\Pi^{\bW}$ is bounded from $L^2(\bW)$ to $L^2(\bV)$ and 
\[  \left \| \Pi^{\bW} \right \|_{L^2(\bW) \rightarrow L^2(\bV)} \le C(d)^{1/2} \fT_1 , \] 
where $C(d)$ is the constant in Theorem \ref{cet}. 
\end{lm}

\begin{rem}
\label{r:parbdd}
The testing condition \eqref{test-03} is clearly  weaker than the testing condition \cond1 from Theorem \ref{t:well-loc-est-02}; the constant $\fT_1$ from \eqref{test-03} is majorized by the corresponding constant from \cond1. 
\end{rem}

\begin{proof}[Proof of Lemma \ref{lem:parbdd}] Fix $f \in L^2(\bW)$ and in the dense set $\mathcal{L}.$ Then by orthogonality,
\begin{align*}
\| \Pi^{\bW} f \|^2_{L^2(\bV)}  = \sum_{Q \in \mathcal{D}} \sum_{R \in \Ch^r(Q)} \left \| \Delta^{\bV}_R  \left( T_{\bW} \E^{\bW}\ci Q f \right) \right \|^2_{L^2(\bV)}. 
\end{align*}
To control this, we use Theorem \ref{cet}. First, for each $Q \in \mathcal{D}$, define the linear map $B\ci Q: \mathbb{F}^d \rightarrow L^2(\bV)$ by 
\[ 
B\ci Q e = \sum_{ R \in \Ch^r(Q)} \Delta^{\bV}_R   T_{\bW} ( \bW(Q)^{-1}\1_Q e ), \quad \forall e \in \mathbb{F}^d,
\]
where $\bW(Q)^{-1}$ is the Moore-Penrose psuedoinverse of $\bW(Q).$ Then defining $A\ci Q := B^*_Q B_Q: \F^d \to \F^d$ we can write
\begin{align*}
\| \Pi^{\bW} f \|^2_{L^2(\bV)} =
\sum_{Q\in \cD} \left\|  A^{1/2}\ci Q \int_Q \dd \bW f \right\|^2 , 
\end{align*}
so we are in position to apply Theorem \ref{cet}. 

To prove condition (ii) in Theorem \ref{cet}, fix $Q_0\in\cD$, $e \in \mathbb{F}^d$ and use the definitions of $A_Q$, $B_Q$ to obtain
\begin{align*}
\sum_{Q\in\cD(Q_0)} \left\|  A^{1/2}\ci Q \bW(Q) e \right\|^2  
&= 
\sum_{Q\in\cD(Q_0)} \left\|  B\ci Q  \bW (Q) e \right\|^2  && \\
&= 
\sum_{Q\in\cD(Q_0)} \sum_{R \in \Ch^r(Q)} \left \| \Delta^{\bV}_R  \left( T_{\bW} \1\ci Q e \right) \right \|^2_{L^2(\bV)}.
\end{align*}
Then  using Lemma \ref{l:T-para} and the testing condition \eqref{test-03} we get
\begin{align*}
 \sum_{Q\in\cD(Q_0)} \left\|  A^{1/2}\ci Q \bW(Q) e \right\|^2 &= 
\sum_{Q\in\cD(Q_0)} \sum_{R \in \Ch^r(Q)}  \left\| \Delta^{\bV}_R  \left( T_{\bW} \1\ci{Q_0} e \right) \right \|^2_{L^2(\bV)}  &&  \text{by Lemma \ref{l:T-para}  } \\
&\le 
\fT_1^2 \|\1\ci{Q_0} e\|\ci{L^2(\bW)}^2 && \text{by \eqref{test-03}},
\end{align*}
so condition \cond2 of Theorem \ref{cet} is verified. Thus
\begin{align*}
 \| \Pi^{\bW} f \|^2_{L^2(\bV)} \le C(d) \fT_1^2 \|f\|\ci{L^2(\bW)}^2 , 
\end{align*}
which completes the proof. 
\end{proof}

\section{Estimates of well localized operators} \label{mainproof}

In this section we will prove Theorem \ref{t:well-loc-est-02}. Theorem \ref{well-loc-rel} will follow automatically, since the bounds $\fT_{1}, \fT_{2}$ and their duals $\fT_{1}^*,\fT_{2}^*$ from Theorem \ref{t:well-loc-est-02} are trivially majorized by the corresponding bounds from Theorem \ref{well-loc-rel}, and the bound $\fT_3$ from  Theorem \ref{t:well-loc-est-02} is dominated by the minimum of $\fT_1$ and $\fT_1^*$ from Theorem \ref{well-loc-rel}.  We will also explain Remark \ref{r:weak}, claiming that the weak estimate \cond3 of Theorem \ref{t:well-loc-est-02} can be relaxed and sometimes ignored. 

To prove  Theorem \ref{t:well-loc-est-02}, we estimate the bilinear form of the operator $T\ci\bW$.  Let $f \in L^2(\bW)$ and $g \in L^2(\bV)$,  with $\|f\|\ci{L^2(\bW)} = \|g\|\ci{L^2(\bV)}=1,$ be from the dense set $\mathcal{L}$ of finite linear combinations of characteristic functions of atoms times vectors, i.e.
\begin{align}
\label{repr-fg}
f = \sum_{j=1}^N  \1\ci{Q_j} e_j\ \text{ and } \ g = \sum_{k=1}^M  \1\ci{R_k}v_k ,
\end{align}
where each  $Q_j, R_k \in \cD$ and  $e_j, v_k \in \F^d.$ By Lemma \ref{dense}, such functions are dense in $L^2(\bW)$ and $L^2(\bV)$ and so to obtain the result, we just need to show that
\begin{align}
\label{main-est-01} 
| \La T\ci{\bW} f, g \Ra\ci{L^2(\bV)}| \le C \| f\|\ci{L^2(\bW)} \| g \|\ci{L^2(\bV)}.
\end{align}

Let us first do some simplifications. Define an equivalence relation $\sim$ on $\cD$, by saying that $Q\sim R$ if $Q$ and $R$ have a common ancestor (i.e.~if $Q, R\subset S$ for some $S\in\cD$).  

Since $T_{\bW}$ is a localized operator, $\La T_{\bW} \1\ci Qe, \1\ci R \nu\Ra\ci{L^2(\bV)} = 0$ if $Q$ and $R$ are in different equivalence classes, for all $e, \nu \in \mathbb{F}^d$. Therefore, it is sufficient to prove \eqref{main-est-01} under the assumption that all $Q_j$, $R_k$ in the representation \eqref{repr-fg} are in the same equivalence class; then taking  the direct sum over equivalence classes, we get the general case. 

Let $Q_0\in\cD$ be a common ancestor of all $Q_j$, $R_k$ appearing in the representation \eqref{repr-fg}.  
Then, by \eqref{eqn:decomp}, we can write $f$, $g$ using the orthogonal decompositions:
\begin{align}
\label{dec-f}
f & = \sum_{Q\in \cD(Q_0)} \Delta\ci Q^{\bW} f + \E\ci{Q_0}^{\bW} f  =: f_1 + f_2 ; \\
\label{dec-g}
g & = \sum_{R\in \cD(Q_0)}\Delta\ci R^{\bV} g + \E\ci{Q_0}^{\bV}g  = :  g_1 + g_2   .
\end{align}
We will estimate the four terms $\La T\ci{\bW} f_j, g_k \Ra_{L^2(\bV)}$  for $1 \le j,k\le 2$ separately.

\subsection{Estimate of the main part}
To estimate $\La T\ci{\bW} f_1, g_1 \Ra\ci{L^2(\bV)}$ let us first notice that by Lemma \ref{lem:parbdd}, the testing condition \cond1 of Theorem \ref{t:well-loc-est-02} and its dual counterpart imply that the paraproducts $\Pi^\bW= \Pi^\bW\ci{T}$ and $\Pi^\bV =\Pi^\bV\ci{T^*}$ are bounded and that 
\begin{align*}
\| \Pi^\bW \|_{L^2(\bW) \rightarrow L^2(\bV)} + \|\Pi^\bV\|_{L^2(\bV) \rightarrow L^2(\bW)} \le C(d)^{1/2} (\fT_1+ \fT_1^*). 
\end{align*}
Thus, it is sufficient to estimate the operator $\wt T\ci\bW := T\ci \bW - \Pi^\bW - (\Pi^\bV)^* 
$. Lemma \ref{lem:replacement} implies that 
\begin{align*}
\Delta\ci R^{\bV}  \wt T\ci{\bW} \Delta\ci Q^{\bW}=
\begin{cases} 
\Delta\ci R^{\bV}  T\ci{\bW} \Delta\ci Q^{\bW}    \,, \qquad &  | \rk Q-\rk R|\le r; \\
0 , & | \rk Q-\rk R| > r, 
\end{cases}
\end{align*}
so
\begin{align*}
\La \wt T\ci{\bW} f_1, g_1 \Ra\ci{L^2(\bV)} 
&= 
\sum_{ \substack{Q, R\in\cD(Q_0) \\ | \rk Q-\rk R|\le r} } \La T\ci{\bW} \Delta\ci Q^{\bW} f, \Delta\ci R^{\bV} g \Ra\ci{L^2(\bV)}  
\\
&= 
\sum_{\substack{Q, R\in\cD(Q_0)\\ \rk Q\le \rk R\le \rk Q +r}}   \La T\ci{\bW} \Delta\ci Q^{\bW} f, \Delta\ci R^{\bV} g \Ra\ci{L^2(\bV)} + \sum_{\substack{Q, R\in\cD(Q_0)\\ \rk R < \rk Q \le \rk R+r}}  \La T\ci{\bW} \Delta\ci Q^{\bW} f, \Delta\ci R^{\bV} g \Ra\ci{L^2(\bV)}. 
\end{align*}
Let us  estimate the first sum. The second one is treated similarly, by considering the dual operator $T\ci\bV^*$. To estimate the first sum, we  need to estimate the operator 
\begin{align*}
\wt T\ci\bW^{+} := \sum_{\substack{Q, R\in\cD(Q_0)\\ \rk Q\le \rk R\le \rk Q +r}} 
\Delta\ci R^{\bV}  T\ci{\bW} \Delta\ci Q^{\bW}  .  
\end{align*}
Since $T_{\bW}$ is $r$-lower triangular, we can see that $\Delta\ci R^{\bV}  T\ci{\bW} \Delta\ci Q^{\bW}=0$ if $\rk R\ge\rk Q$ and $R\not\subset Q^{(r)}$. So, we can rewrite $\wt T\ci\bW^{+}$ as 
\begin{align*}
\wt T\ci\bW^{+} & = \sum_{S\in\cD (Q_0^{(r)})}  \sum_{\substack{R\in\cD(S) \cap \cD(Q_0)\\ r+\rk S\le \rk R\le \rk S +2r}}  \sum_{Q\in\Ch^rS \ \cap \cD(Q_0)}
\Delta\ci R^{\bV}  T\ci{\bW} \Delta\ci Q^{\bW} =: \sum_{S\in \cD (Q_0^{(r)})} \wt T\ci\bW^{+, S},
\intertext{where}
\wt T\ci\bW^{+, S} & =   \sum_{\substack{R\in\cD(S)\cap \cD(Q_0)  \\ r+\rk S  \le \rk R\le \rk S +2r}} \sum_{Q\in \Ch^rS \ \cap \cD(Q_0)}
\Delta\ci R^{\bV}  T\ci{\bW} \Delta\ci Q^{\bW} .
\end{align*}
The testing condition \cond2 of Theorem \ref{t:well-loc-est-02} implies that 
\begin{align}
\label{est-T^S}
\| \wt T\ci\bW^{+, S}\|\ci{L^2(\bW)\to L^2(\bV)} \le \fT_2.
\end{align}
Note that if $S\cap S'=\varnothing$ or $|\rk S -\rk S'|>r$ then 
\begin{align*}
\ran \wt T\ci\bW^{+, S} \perp \ran\wt T\ci\bW^{+, S'}, \qquad \left(\ker \wt T\ci\bW^{+, S} \right)^\perp \perp \left(\ker \wt T\ci\bW^{+, S'} \right)^\perp .
\end{align*}
Therefore for fixed $k\in \Z$, the operator $\wt T\ci\bW^{+, k} $ defined by
\begin{align*}
\wt T\ci\bW^{+, k} := \sum_{j\in\Z} \sum_{\substack{S\in \cD(Q_0^{(r)}) \\ \rk S = k + (r+1)j}} \wt T\ci\bW^{+, S}
\end{align*}
is the direct sum of the corresponding operators  $\wt T\ci\bW^{+, S}$, and the estimate \eqref{est-T^S} implies 
\begin{align}
\label{est-T^k}
\| \wt T\ci\bW^{+, k} \|\ci{L^2(\bW)\to L^2(\bV)} \le \fT_2 . 
\end{align}
Since 
$
\wt T\ci \bW^{+} = \sum_{k=0}^r \wt T\ci\bW^{+, k}
$, 
we can easily conclude from \eqref{est-T^k} that 
\begin{align*}
\| \wt T\ci\bW^{+} \|\ci{L^2(\bW)\to L^2(\bV)} \le (r+1) \fT_2 . 
\end{align*}
However, by being more careful, we can obtain the following better dependence on $r$:
\begin{align}
\label{est-T^k-01}
\| \wt T\ci\bW^{+} \|\ci{L^2(\bW)\to L^2(\bV)} \le (r+1)^{1/2} \fT_2 .
\end{align}
To get this, observe that for for $0\le j <k \le r$ 
\begin{align*}
\left( \ker \wt T\ci\bW^{+,j} \right)^\perp  \perp  \left( \ker \wt T\ci\bW^{+, k} \right)^\perp.
\end{align*}
Then, decomposing $f_1=\sum_{k=0}^r f^k$, where
\begin{align*}
f^k:= \sum_{n\in\Z} \sum_{\substack{S\in\cD(Q_0^{(r)})\\ \rk S=k + (r+1)n}} \sum_{Q\in\Ch^r S} \Delta\ci Q^\bW f, 
\end{align*}
we get  that 
\begin{align*}
\| \wt T\ci\bW^{+} f_1\|\ci{L^2(\bV)} & = \left\|  \wt T\ci\bW^{+}\sum_{k=0}^r f^k \right\|_{L^2(\bV)} 
=  \left\|  \sum_{k=0}^r \wt T\ci\bW^{+,k} f^k \right\|_{L^2(\bV)} \\
& \le \fT_2\sum_{k=0}^r \| f^k\|\ci{L^2(\bW)} 
\le\fT_2 (r+1)^{1/2} \left( \sum_{k=0}^r \| f^k\|\ci{L^2(\bW)}^2 \right)^{1/2} ; 
\end{align*}
here the last inequality is by Cauchy--Schwarz. 

\subsection{Estimates of parts involving constant functions}
Estimates 
\begin{align*}
|\La T\ci \bW f_2 , g_1 \Ra\ci{L^2(\bV)} | \le \fT_1, \qquad 
|\La T\ci \bW f_1 , g_2 \Ra\ci{L^2(\bV)} | \le \fT_1^*
\end{align*}
follow immediately from the testing condition \cond1 and its dual. Estimate 
\begin{align*}
|\La T\ci \bW f_2 , g_2 \Ra\ci{L^2(\bV)} | \le \fT_3
\end{align*}
is a direct corollary of the assumption \cond3. 

Note that in decompositions \eqref{dec-f} and \eqref{dec-g}, we can replace $Q_0$ by any of its ancestors, so, as we said in 
Remark \ref{r:weak} it is sufficient that the estimate \cond3  holds  only for sufficiently large cubes $Q$ (meaning that for any $Q_0\in\cD$ we can find $Q\in\cD$, $Q_0\subset Q$ such that \cond3 holds for $Q$). 

Moreover, if for the increasing sequence of cubes $Q_n$, $n\ge 0$, where $Q_{n+1} $ is the parent of $Q_n$, we have that $\bW(Q_n)\ge \alpha_n \I$, $\alpha_n\nearrow \infty$ then writing decomposition \eqref{dec-f} with $Q_n$ instead of $Q_0$ and letting $n\to +\infty,$ we obtain
\begin{align*}
f = \sum_{Q\in \cD} \Delta\ci Q^{\bW} f =: f_1. 
\end{align*}
The analogous condition for $\bV$ implies the similar representation for $g$, so the theorem is reduced to estimating $\La T\ci\bW f_1, g_1\Ra\ci{L^2(\bV)}$, which was done using only testing conditions \cond1, \cond2 and their duals.

\section{Estimates of the Haar shifts}
\label{s:est-Haar}

In this section we will prove Lemmas \ref{l:TestHaar} and \ref{l:TQW-TWQ}. Theorem \ref{band-rel} is then a simple corollary of Theorem \ref{t:well-loc-est-02}. 
We will need the following lemma, which is well known to specialists; for the convenience of the reader we present its proof here. 
\begin{lm}
\label{l:norm<A_2}
Let $T$ be an integral operator with kernel $K$, $Tf(x) = \int K(x,y) f(y) d\sigma(y)$, where $K$ is supported on $Q\times Q$ ($Q\in\cD$) and $\|K\|_\infty\le |Q|^{-1}$. If the $d \times d$ matrix measures $\bV$, $\bW$ satisfy the matrix $A_2$ condition \eqref{A_2}, then the operator $T\ci\bW$, $T\ci \bW f := T(\bW f)$ satisfies
\begin{align*}
\| T\ci\bW \|\ci{L^2(\bW)\to L^2(\bV)} \le d^{1/2} [\bV, \bW]^{1/2}\ci{A_2}\,.
\end{align*}
\end{lm}

\begin{proof}
Take $f\in L^2(\bW)$, $g\in L^2(\bV)$,  with $\|f\|\ci{L^2(\bW)} = \|g\|\ci{L^2(\bV)} =1$. As we discussed above in Section \ref{s:mvm}, we can assume without loss of generality that the measures $\bV$ and $\bW$ are absolutely continuous with respect to the scalar trace measures $\bv$ and $\bw$ respectively, $\dd \bV = V\dd \bv$, $\dd \bW= W \dd\bw$. We then can write
\begin{align*}
\left| \La T\ci\bW f, g\Ra\ci{L^2(\bV)} \right| 
 \le  \iint_{Q\times Q} 
\left| \La V(x) K(x,y) W(y) f(y), g(x) \Ra\ci{\F^d} \right| \dd\bv(x) \dd\bw(y).
\end{align*}
The integral then can be estimated by 
\begin{align*}
  |Q|^{-1} 
&  \iint_{Q\times Q} \|V^{1/2}(x) W^{1/2}(y) \| \ \| V^{1/2}(x) g(x)\|\ci{\F^d} 
\| W^{1/2}(y) f(y) \|\ci{\F^d} \dd\bv(x) \dd\bw(y) \\
&\le \left(\iint_{Q\times Q} \| V^{1/2}(x) g(x)\|\ci{\F^d}^2 
\| W^{1/2}(y) f(y) \|\ci{\F^d}^2 \dd\bv(x) \dd\bw(y) \right)^{1/2}\times \\
& \qquad \qquad \qquad \qquad\qquad \times 
\left( |Q|^{-2}\iint_{Q\times Q} \|V^{1/2}(x) W^{1/2}(y) \|^2 \dd\bv(x) \dd\bw(y) \right)^{1/2} \\
& = 
\|f\|\ci{L^2(\bW)} \|g\|\ci{L^2(\bV)} 
\left( |Q|^{-2}\iint_{Q\times Q} \|V^{1/2}(x) W^{1/2}(y) \|^2 \dd\bv(x) \dd\bw(y) \right)^{1/2}\,.
\end{align*}
In the last integral, we can replace the operator norm by the Frobenius (Hilbert--Schmidt) norm $\|\fdot\|\ci{\fS_2}$ (recall that $\|A\|\ci{\fS_2}^2 = \tr(A^*A)$): 
\begin{align*}
\iint_{Q\times Q} \|V^{1/2}(x) W^{1/2}(y) \|^2 \dd\bv(x) \dd\bw(y) 
&\le 
\iint_{Q\times Q} \|V^{1/2}(x) W^{1/2}(y) \|\ci{\fS_2}^2 \dd\bv(x) \dd\bw(y) \\
& = \iint_{Q\times Q} \tr \Bigl(V(x) W(y) \Bigr) \dd\bv(x) \dd\bw(y) \\
&= \tr \Bigl(\bV(Q) \bW(Q) \Bigr) \\
& = \| \bV(Q)^{1/2} \bW(Q)^{1/2} \|\ci{\fS_2}^2 \\
& \le d \ \| \bV(Q)^{1/2} \bW(Q)^{1/2} \|^2 \\
& \le d  \ |Q|^2 [\bV, \bW]\ci{A_2}. 
\end{align*}
Combining this with the previous estimate, we get the conclusion of the lemma.  
\end{proof}

\subsection{Comparison of different truncations}
Let $T$ be a generalized big Haar shift of complexity $r$. As before, we will assume that the we only have finitely many terms $T\ci Q$ in the representation \eqref{BigHaarShift} and that each block $T\ci{ Q}$ is represented by an integral operator with a bounded kernel.

In the testing conditions in Theorems \ref{t:well-loc-est-02} and \ref{band-rel}, we used different truncations of the operator $T\ci\bW$, namely $T\ci\bW^Q$ and $(T^Q)\ci\bW$ respectively.  These operators  are generally different, but their difference can be estimated. 

Now to state the estimate, we will need some new notation. Let  $P\ci Q^\bV$  be the orthogonal projection in $L^2(\bV)$ onto the subspace of functions supported on $Q$ and orthogonal to $\{\1\ci Q e:e\in\F^d\}$. Then, by \eqref{eqn:decomp}, 
\begin{align*}
P\ci Q^\bV f = \sum_{R\in\cD(Q)} \Delta\ci R^\bV f = \1\ci Q f - \E\ci Q^\bV f, \qquad \forall f \in \mathcal{L},
\end{align*}
and we can extend this to all $f \in L^2(\bW).$
Then in this notation, the operator $T^Q\ci\bW$ defined above can be written as $T^Q\ci\bW=P\ci Q^\bV T\ci\bW.$
\begin{lm}
\label{l:TQW-TWQ}
For operators $T^Q\ci\bW$ and $(T^Q)\ci\bW$ introduced above and $f \in L^2(\bW)$ supported on $Q$,
\begin{align*}
\left\| \left(  T^Q\ci\bW - P\ci Q^\bV (T^Q)\ci\bW  \right) f  \right\|\ci{L^2(\bV)} \le d^{1/2} r  [\bV, \bW]\ci{A_2}^{1/2} \|f\|\ci{L^2(\bW)} \, .
\end{align*}
\end{lm}

\begin{proof}
For $f \in \mathcal{L}$ and supported on $Q$ we have 
\begin{align*}
\left(  T^Q\ci\bW - P\ci Q^\bV (T^Q)\ci\bW  \right) f = \sum_{k=1}^r P\ci Q^\bV T\ci{Q^{(k)}}(\bW f),
\end{align*}
where $Q^{(k)}$ is the ancestor of $Q$ order $k$. Note that the terms $T\ci{Q^{(k)}}(\bW f)$ with $k>r$ are annihilated by $P\ci Q^\bV$.

Each operator $T\ci{Q^{(k)}}$ is an integral operator with kernel $K\ci{Q^{(k)}}$ supported on $ Q^{(k)} \times Q^{(k)} $ and satisfying $\|K\ci{Q^{(k)}}\|_\infty \le |Q^{(k)}|^{-1}$. Therefore applying Lemma \ref{l:norm<A_2} and using the fact that $P\ci Q^\bV$ is an orthogonal projection (and so a contraction) in $L^2(\bV)$ we get 
\begin{align*}
\| P\ci Q^\bV T\ci{Q^{(k)}}(\bW f)\|\ci{L^2(\bV)} \le d^{1/2} [\bV, \bW]^{1/2}\ci{A_2}\, \|f\|\ci{L^2(\bW)}\,. 
\end{align*}
Summation over $k$ completes the proof. 
\end{proof}

\begin{lm}
\label{l:nec} Assume there exists $\kappa\in (0,1)$ such that for any $Q\in\cD$ and its parent $\hat Q,$
\[
|Q|\le\kappa |\hat Q|. 
\]
Further, assume the weights $\bV$, $\bW$ satisfy the matrix $A_2$ condition  \eqref{A_2}. 
Then for any generalized big Haar shift $T$ of order $r$ with finitely many non-zero terms and for any $f\in L^2(\bW)$ supported on $Q\in\cD$
\[
\left\| \1\ci Q \left(T\ci\bW - (T^Q)\ci\bW \right) f \right\|\ci{L^2(\bV)} \le (1-\kappa)^{-1} d^{1/2} [\bV, \bW]\ci{A_2}^{1/2} \| f \|\ci{L^2(\bW)}. 
\]
 \end{lm}
 
\begin{cor}
\label{c:nec}
Under the assumptions of Lemma \ref{l:nec}, the condition \cond1 from Theorem  \ref{band-rel} is necessary for the boundedness of the operator $T\ci\bW: L^2(\bW)\to L^2(\bV)$, and the constants  $\fT$, $\fT^*$ from Theorem \ref{band-rel} satisfy
\[
\fT, \fT^* \le \|T\|\ci{L^2(\bW)\to L^2(\bV)} + (1-\kappa)^{-1} d^{1/2} [\bV, \bW]\ci{A_2}^{1/2}. 
\]

\end{cor}

\begin{proof}[Proof of Lemma \ref{l:nec}] Fix $Q \in \mathcal{D}$ and let $Q^{(k)} $ be the ancestor of order $k$ of $Q$. Then for $f$ supported on $Q$
\[
\left(T\ci\bW - (T^Q)\ci\bW \right) f =\sum_{k\ge 1} T\ci{Q^{(k)}} (\bW f).
\]
This sum can be written as the integral operator $\int_Q K(x,y) \bW(y) f(y) \dd y$, where 
\[
K(x,y) = \sum_{k\ge 1} K\ci{Q^{(k)}} (x,y),  
\]
and $K\ci{Q^{(k)}} $ is the kernel of the integral operator $T\ci{Q^{(k)}}$. We can now apply Lemma \ref{l:norm<A_2} to this sum of integral operators to obtain the desired bound.
\end{proof}
\subsection{Proof of Lemma \ref{l:TestHaar}}
Let the testing condition \eqref{test-T^Q} hold. 
Applying Lemma \ref{l:TQW-TWQ} with $f=\1\ci Q e$ and noticing that 
\[
\| P\ci Q^\bV (T^Q)\ci\bW f \|\ci{L^2(\bV)} \le \|  (T^Q)\ci\bW f \|\ci{L^2(\bV)} \le\fT \|f\|\ci{L^2(\bW)},
\]
we immediately get 
\eqref{test-T_W^Q}. 

To get \eqref{test-T_W^Q-01}, we need a bit more work. We can write 
\[
T^Q = T^{r+1} + \sum_{k=0}^r T_k, 
\]
where 
\begin{align*}
T^{r+1} = \sum_{R\in \Ch^{r+1} Q} T^R , \qquad T_k = \sum_{R\in\Ch^k Q} T\ci R, 
\end{align*}
with the obvious agreement that $\Ch^0 Q=\{Q\}$.  Following the agreed notation, for a scalar  integral operator $T$ we denote by $T\ci \bW$ the operator defined by 
\[
T\ci \bW f := T(\bW f), 
\] 
whenever this expression is defined. 

The operators $T\ci R$ are $R$-localized, meaning that $T\ci R f= T\ci R(\1\ci R f)$, and  $T\ci R f$ is supported on $R$, and the same holds for $T^R$.  

The functions $f_Q\in D\ci Q^{\bW, r}$ are constant on cubes $R\in\Ch^{r+1} Q$, so using the testing condition \eqref{test-T^Q} and the fact that the operators $T^R$ are $R$-localized, we get for $f_Q\in D\ci Q^{\bW, r}$
\begin{align}
\notag
\| (T^{r+1})\ci\bW f_Q \|\ci{L^2(\bV)}^2 & =  \sum_{R\in\Ch^{r+1} Q} \| T^R(\bW \1\ci R f_Q) \|\ci{L^2(\bV)}^2 \\
\notag
&\le \sum_{R\in\Ch^{r+1} Q} \fT^2 \| \1\ci R  f_Q \|\ci{L^2(\bW)}^2 \\
& \label{est-T^{r+1}} = \fT^2 \| f_Q \|\ci{L^2(\bW)}^2.
\end{align}
To estimate the operators $T_k,$ we estimate each block $T\ci R$ by Lemma \ref{l:norm<A_2}, and using the fact that $T\ci R$ is $R$-localized we get for $f_Q\in D\ci Q^{\bW, r}$
\begin{align*}
\| (T_k)\ci\bW f_Q \|\ci{L^2(\bV)}^2 & = \sum_{R\in\Ch^{k} Q} \| T\ci R(\bW \1\ci R f_Q) \|\ci{L^2(\bV)}^2  \\
&\le \sum_{R\in\Ch^{k} Q} d \ [\bV, \bW]\ci{A_2} \| \1\ci R  f_Q \|\ci{L^2(\bW)}^2 \\
& = d \ [\bV, \bW]\ci{A_2}  \|  f_Q \|\ci{L^2(\bW)}^2  \,.
\end{align*}
Adding these estimates for $k=0,1,\ldots, r$ and combining them with \eqref{est-T^{r+1}}, we see that for any $f_Q\in D\ci Q^{\bW,r}$ 
\begin{align*}
\| (T^Q)\ci\bW f_Q \|\ci{L^2(\bV)}  \le
\left( d^{1/2}  (r+1) [\bV, \bW]\ci{A_2}^{1/2} + \fT \right) \|f_Q \|\ci{L^2(\bW)} . 
\end{align*}
Since the projection $P\ci Q^\bV$ is a contraction in $L^2(\bV)$, the same estimate holds for the norm $\| P\ci Q^\bV (T^Q)\ci\bW f \|\ci{L^2(\bV)}$, so combining it with Lemma \ref{l:TQW-TWQ}, we obtain
\eqref{test-T_W^Q-01}.

Finally, to show that \eqref{weak-02} holds, let us recall that $T=\sum_{R\in\cR} T\ci R$, where $\cR\subset \cD$ is some finite collection. Then for each $Q_0\in\cD$ we can find a cube $Q\supset Q_0$ which is not contained in any $R\in\cR$. Then $ T\ci \bW \1\ci Q e = (T^Q)\ci\bW \1\ci Q e $, and \eqref{weak-02}  follows from \eqref{test-T^Q}.

\section{Appendix: density of simple functions}

\begin{lm}\label{dense}
Let $\cF$  be the smallest $\sigma$-algebra containing an increasing sequence of atomic $\sigma$-algebras $\cF_n$, with sets of atoms $\mathcal{D}_n.$ Let $\mathcal{L}$
denote the space of linear combinations of functions $\textbf{1}_{Q}e$ with $Q\in \cD = \cup_{n} \mathcal{D}_n$ and $e \in \F^d$. If $\bW$ is a $d\times d$ matrix valued
measure defined on $\cF$, then $\mathcal{L}$ is dense in $L^2(\bW)$.
\end{lm}

\begin{proof} First, observe that the result is true for scalar measures. Indeed, if $\sigma$ is a scalar measure defined on $\cF$,  then  linear combinations of sets $\textbf{1}_{Q}$, $Q \in \mathcal{D}$ are dense in $L^2(\sigma)$. To see this, observe that we can obtain $\sigma$ by first starting with $\sigma$ defined on $\mathcal{D}$ and then extending $\sigma$ to $\mathcal{P}(\mathcal{X})$ via the outer measure
\begin{equation} \label{eqn:outer} \sigma^*(F) :=  \text{inf} \left \{ \sum_{j=1}^{\infty} \sigma(Q_j): Q_j \in \mathcal{D}, F \subset \bigcup_{j=1}^{\infty} Q_j \right \}.  \end{equation}
Then the Carath\'eodory's Theorem implies that $\sigma^*$ restricts to a measure on a $\sigma$-algebra $\mathcal{M}$, which contains $\mathcal{F}$. We only consider this measure restricted to $\mathcal{F}$
 and uniqueness implies that this measure is $\sigma$. Then \eqref{eqn:outer} shows linear combinations of
  $\textbf{1}_Q$, $Q \in \mathcal{D}$ are dense in the set of linear combinations of $\textbf{1}_F,$ $F \in \mathcal{F}$, which are dense in $L^2(\sigma).$

Now consider the matrix setting. Let $\bW$ be a $d\times d$ matrix valued
measure defined on $\cF$ with trace measure $\bw := \sum \bw_{i,i}$. Then 
\[ | \bw_{i,j}(F)| \le d \bw (F) \qquad \forall  F \in \mathcal{F}.\]
Thus, the Radon-Nikodym Theorem allows us to write $\bW = W(x)  \ d \bw$, where the entries of $W$ are in $L^{\infty}(\bw).$

We claim that if $f\in L^2(\bW)$ satisfies $\La f, e \textbf{1}_Q\Ra\ci{L^2(\bW)}=0$ for all $e\in \F^d$, $Q \in \cD,$ then $f \equiv0$ in $L^2(\bW).$
To see this, fix $e \in \F^d$ and suppose that for any $Q\in \cD$,
\[ \int_{\cX} \left \La W(x) f , e \right \Ra_{\mathbb{F}^d} \1\ci Q  \ d\bw
 = \int_{\cX} \left \La d \bW f , \1\ci Q e \right \Ra_{\mathbb{F}^d} =
0.
\]
The scalar result implies the function $\La W(x)f(x),e\Ra_{\mathbb{F}^d} =0, \bw$-a.e. 
%
%
Let $W(x)^{-1}$ denote the Moore-Penrose pseudoinverse of $W(x).$ Then 
\[W(x)^{-1} W(x)f(x) =0 \ \bw\text{-}a.e.\] 
as well. This immediately implies that 
\[  \| f \|_{L^2(\bW)}^2 = \int_{\cX} \left \langle d \bW f, f \right \rangle  =  \int_{\cX} \left \langle W^{-1} W f, W f \right \rangle  d\bw =0,\]
so $f$ is the zero element in the space $L^2(\bW)$.
\end{proof}

\end{document}